\documentclass[11pt]{article}

\usepackage[margin=1in]{geometry}
\usepackage[T1]{fontenc}
\usepackage{mathpazo}
\usepackage{microtype}
\usepackage{setspace}
\usepackage{amsmath,amssymb,amsthm,mathtools}
\usepackage{bm}

\usepackage{graphicx}
\usepackage{booktabs}
\usepackage{tikz}
\usetikzlibrary{positioning,arrows.meta,calc}

\usepackage[numbers,sort&compress]{natbib}
\usepackage[colorlinks=true,linkcolor=blue!50!black,citecolor=blue!50!black,urlcolor=blue!50!black]{hyperref}
\usepackage[nameinlink,capitalize,noabbrev]{cleveref}

\usepackage{enumitem}
\setlist[itemize]{leftmargin=*,topsep=2pt,itemsep=2pt}
\setlist[enumerate]{leftmargin=*,topsep=2pt,itemsep=2pt}

\theoremstyle{plain}
\newtheorem{theorem}{Theorem}[section]
\newtheorem{proposition}[theorem]{Proposition}

\newtheorem{corollary}[theorem]{Corollary}

\theoremstyle{definition}
\newtheorem{definition}[theorem]{Definition}
\newtheorem{assumption}[theorem]{Assumption}
\newtheorem{example}[theorem]{Example}

\theoremstyle{remark}
\newtheorem{remark}[theorem]{Remark}

\DeclareMathOperator{\E}{\mathbb{E}}
\DeclareMathOperator{\Prob}{\mathbb{P}}
\DeclareMathOperator{\Var}{Var}
\DeclareMathOperator{\KL}{KL}

\newcommand{\R}{\mathbb{R}}
\newcommand{\N}{\mathbb{N}}
\newcommand{\X}{\mathcal{X}}
\newcommand{\F}{\mathcal{F}}

\newcommand{\1}{\mathbf{1}}

\newcommand{\loss}{\ell}
\newcommand{\woe}{W} 
\newcommand{\evar}{E} 

\newcommand{\acks}[1]{\section*{Acknowledgments}#1}
\newcommand{\BlackBox}{\rule{1.5ex}{1.5ex}}

\numberwithin{equation}{section}

\title{\Large\bfseries Bayes, E-values, and Testing}
\author{Nicholas G. Polson\thanks{Booth School of Business, University of Chicago. \texttt{ngp@chicagobooth.edu}}
\and Vadim Sokolov\thanks{Department of Systems Engineering and Operations Research, George Mason University. \texttt{vsokolov@gmu.edu}}
\and Daniel Zantedeschi\thanks{School of Information Systems, Muma College of Business, University of South Florida. \texttt{danielz@usf.edu}}}
\date{}

\begin{document}
\maketitle

\begin{abstract}%
E-values and E-processes (nonnegative supermartingales) provide
anytime-valid evidence for sequential testing via Ville's inequality,
yet their connection to Bayesian reasoning, representational structure,
and computational feasibility are often conflated in the literature.
We develop a typed framework that separates sequential evidence into
three layers: (i)~\emph{representation} (Radon--Nikod\'ym / likelihood-ratio
geometry), (ii)~\emph{validity} (supermartingale certificates under optional
stopping), and (iii)~\emph{decision} (boundary design and efficiency calibration).
Our main results are:
(a)~under log-loss and Bayes-risk minimization, the likelihood ratio is
the unique evidence representation within the coherent predictive subclass
(Theorem~\ref{thm:canonical_lr});
(b)~the likelihood-ratio stopping time satisfies
$\E_1[\tau_b] = (\log b)/\mu + O(\sqrt{\log b})$ under Cram\'er conditions,
while validity-only thresholds admit no such growth-rate guarantee
(Theorem~\ref{thm:moderate_deviation}, Proposition~\ref{prop:validity_vs_lr});
and (c)~regret-optimal codes (e.g., NML/MDL) do not in general yield valid
E-processes, while prequential codes do
(Proposition~\ref{prop:comp_boundary}).
Monte Carlo experiments confirm the theoretical predictions.
The framework applies to online model validation, adaptive experimentation,
conformal prediction, and sequential changepoint detection.
\end{abstract}


\section{Introduction}
\label{sec:intro}

A deployed machine learning system generates predictions continuously.
A classifier monitoring patient risk scores, an adaptive A/B test
allocating traffic to treatments, a conformal predictor issuing
coverage-guaranteed prediction sets: each accumulates data sequentially
and may be stopped, audited, or updated at any time.
Classical fixed-sample inference (p-values, confidence intervals) loses
its guarantees under such optional stopping \citep{Wald1947,JennisonTurnbull2000}.
How can we accumulate evidence over time without invalidating error control
under arbitrary, data-dependent stopping rules?

E-processes, nonnegative supermartingales with unit initial
expectation, provide a principled answer.
Ville's inequality \citep{Ville1939} guarantees that
$\Prob_{H_0}(\sup_t E_t \ge c) \le 1/c$ for any stopping time,
delivering anytime-valid Type~I control without
$\alpha$-spending or sample-size commitments.
This framework has found applications in safe testing
\citep{GrunwaldEtAl2024}, adaptive experimentation
\citep{HowardEtAl2021, WaudbySmithRamdas2024},
game-theoretic probability \citep{ShaferVovk2019},
and conformal prediction \citep{Vovk2020conformal}.

Despite these advances, three questions remain open.
First, when does an E-process admit a likelihood-ratio or
generalized-likelihood-ratio representation (including composite null
constructions such as numeraire E-variables), and what forces this
structure?
Second, how do validity-only thresholds (Markov/Ville) compare with
likelihood-ratio-optimal boundaries in terms of statistical efficiency?
Third, which computational objects (codes, description lengths,
regret-optimal predictors) can be converted into valid E-processes,
and which provably cannot?

These questions are difficult to address simultaneously because
existing treatments tend to blur the distinction between \emph{what}
an evidence measure is (a likelihood ratio? a betting score? a
code-length difference?), \emph{why} it is valid (supermartingale
property? Kraft inequality? exchangeability?), and \emph{how} it is
used (fixed threshold? sequential boundary? Bayes-risk-optimal
cutoff?). Conflating these roles leads to confusion in practice:
a code-length function can look like an E-value without being one
\citep{Grunwald2007}, and a valid E-process can have zero statistical
power if its boundary is chosen without regard to the underlying
representation \citep{RamdasEtAl2023}.
The remedy is a modular decomposition that keeps each role
separate, much as the bias-variance decomposition separates
estimation error from model complexity, or as the PAC learning
framework separates sample complexity from hypothesis class
expressiveness \citep{ShaferVovk2019}.

This paper is not a survey of e-values; it contributes new canonicality,
moderate-deviation stopping, and code-to-e obstruction results, organized
by a typed interface that makes their logical interdependence precise.
We address all three questions through a framework that separates
representation, validity, and decision into formally distinct layers.

\subsection{The Probabilistic Landscape}
\label{sec:landscape}

Table~\ref{tab:landscape} summarizes how three classical
structures describe the information in a sample path:
Sanov's large-deviation rate, the E-process growth rate
$n^{-1}\log E_n \to D_{\KL}(\mu \| P_0)$
(Proposition~\ref{prop:e_growth}), and posterior contraction.
Only the E-process column is needed for the main development.
The key efficiency distinction is the gap between
calibration-only $1/c$ control (Markov/Ville) and
representation-aware exponential detection at rate
$(\log b)/D_{\KL}$
\citep{Cover1969, HellmanCover1970, Lindley1961};
this gap is formalized in
Theorem~\ref{thm:moderate_deviation} and
Proposition~\ref{prop:validity_vs_lr}.
The broader probabilistic landscape (de~Finetti, inverse Sanov,
martingale posteriors, deviation regimes) is self-contained in
Appendix~\ref{app:landscape}.

\begin{table}[t]
\centering
\caption{Three frameworks describing the information in a single sample path
about the directing measure $\mu$.}
\label{tab:landscape}
\small
\begin{tabular}{@{}llll@{}}
\toprule
\textbf{Framework} & \textbf{Primary object} & \textbf{Path convergence} & \textbf{Rate} \\
\midrule
Sanov / LDP & Empirical measure $L_n$ & $D_{\KL}(\cdot \| P_0)$ rate fn & Exponential \\
E-process & Test martingale $E_n$ & $n^{-1}\!\log E_n \to D_{\KL}(\mu \| P_0)$ & Linear in $n$ \\
Mart.\ posterior & $\Pi_n$ (measure-valued) & $\Pi_n \to \delta_\mu$ a.s. & Posterior contraction \\
\bottomrule
\end{tabular}
\end{table}

\subsection{Contributions}
\label{sec:contributions}

We establish the following results.

\begin{enumerate}
\item \textbf{Canonicality under log-loss (Theorem~\ref{thm:canonical_lr}).}
Under coherent prediction and log-loss Bayes risk, the Fubini/tower
decomposition identifies the likelihood ratio as the unique canonical
evidence representation.
The Bayes-risk-optimal test is a threshold rule on the likelihood-ratio
process; general E-process constructions are valid but need not recover
this optimal rejection region.

\item \textbf{Moderate-deviation stopping boundary
(Theorem~\ref{thm:moderate_deviation}, Proposition~\ref{prop:validity_vs_lr}).}
Under explicit Cram\'er conditions on the log-likelihood increments,
we prove that the LR stopping time satisfies
$\E_1[\tau_b] = (\log b)/\mu + O(\sqrt{\log b})$
with $(\tau_b - (\log b)/\mu)/\sqrt{\log b} = O_p(1)$.
A companion structural separation result shows that generic
E-processes lacking LR structure admit no exponential growth-rate
characterization, confining them to the $O(1/b)$ calibration scale.

\item \textbf{Computational obstruction
(Proposition~\ref{prop:comp_boundary}).}
Regret-optimal codes (NML/MDL) do not in general yield valid E-processes:
their normalizing constants depend on the full sample size, violating
the sequential factorization required for the supermartingale property.
We characterize sufficient conditions for conversion (prequential codes)
and identify the structural boundary between coding-theoretic and
probabilistic evidence.

\item \textbf{Evidence-class algebra and maximality
(Theorem~\ref{thm:class_algebra_informal}, Proposition~\ref{prop:maximality}).}
The class of E-processes forms a convex set closed under scaling by
$c \in (0,1]$, predictable stopping, countable mixtures, and Bayesian
marginalization, and is the \emph{largest} such class preserving Ville control.
These compositional properties support modular construction of evidence
in online pipelines.

\item \textbf{Scoring-rule uniqueness
(Proposition~\ref{prop:log_uniqueness_scoring}).}
Among strictly proper scoring rules, log-loss is the unique rule whose
induced evidence ratios form supermartingales.
This conceptual boundary theorem delineates the scope of the typed framework.

\item \textbf{Conformal e-prediction
(Proposition~\ref{prop:conformal_e}).}
Under exchangeability, nonconformity-based E-values provide anytime-valid
coverage guarantees for sequential prediction, connecting the typed
framework to distribution-free online learning.
\end{enumerate}

\subsection{Related Work}
\label{sec:related}

Our framework builds on several lines of research.

\noindent\textit{E-values and safe testing.}
The modern E-value framework originates with \citet{VovkWang2021} and
\citet{GrunwaldEtAl2024}, who formalize E-variables as calibrated
measures of evidence with anytime-valid guarantees.
\citet{RamdasEbook} provides a comprehensive treatment of e-processes,
including the distinction between test supermartingales and general
e-processes and constructions for composite testing.
For composite nulls, the numeraire and reverse information projection
\citep{LarssonRamdasRuf2025} provide a canonical representation-layer
object $E^\star$, which fits directly into our typed interface.
Our contribution is orthogonal: rather than constructing new
E-processes, we characterize the \emph{interfaces} between layers
(representation$\to$validity, validity$\to$decision), identifying when
likelihood-ratio structure is forced and quantifying the efficiency
gap when it is absent.

\noindent\textit{Sequential testing and confidence sequences.}
\citet{HowardEtAl2021} develop time-uniform confidence sequences via
sub-$\psi$ conditions; \citet{WaudbySmithRamdas2024} construct
confidence sequences by betting; \citet{KaufmannKoolen2021} analyze
mixture martingales for sequential tests via hierarchical priors.
These methods operate at the validity layer of our framework;
our moderate-deviation stopping theorem (Theorem~\ref{thm:moderate_deviation})
complements them by quantifying the gap between Markov/Ville
and likelihood-ratio calibration.

\noindent\textit{Online learning and adaptive inference.}
Adaptive data analysis requires evidence that remains valid under
data-dependent decisions.
\citet{GrunwaldEtAl2024} connect E-values to always-valid
$p$-values; \citet{RamdasEtAl2023} develop game-theoretic testing.
Our computational obstruction result
(Section~\ref{sec:computational_limits}) is relevant to online model
selection via MDL \citep{Grunwald2007}, showing when code-based evidence
fails sequential validity.

\noindent\textit{Conformal prediction.}
\citet{VovkGammerShaf2005} introduce conformal prediction;
\citet{Vovk2020conformal} develops conformal e-prediction.
\citet{GibbsCandesConformal2024} extend conformal methods to
distribution shift; \citet{OliveiraEtAl2024} extend split conformal
prediction to non-exchangeable data with explicit coverage penalties.
Our typed separation clarifies the relationship between conformal
coverage (marginal validity) and e-process control
(supermartingale validity); see Section~\ref{sec:conformal}.

\noindent\textit{Coding and MDL.}
\citet{Cover1985} identifies the deep link between Kolmogorov complexity,
data compression, and statistical inference.
\citet{Rissanen1978} and \citet{Grunwald2007} develop the MDL principle;
\citet{Shtarkov1987} introduces NML.
\citet{Dawid1984} proposes the prequential principle.
Our computational boundary theorem formalizes why NML codes fail as
E-processes while prequential codes succeed.

\section{Sequential Evidence Framework}
\label{sec:framework}

We formalize the mathematical objects and their relationships.
All definitions require a filtered probability space
$(\Omega, \F, (\F_t)_{t \ge 0}, \Prob_{H_0})$ under the null hypothesis.

\subsection{E-Variables and E-Processes}

\begin{definition}[E-variable]
\label{def:evar}
A nonnegative random variable $\evar$ is an \emph{E-variable} for $H_0$ if
$\E_{H_0}[\evar] \le 1$.
\end{definition}

\begin{definition}[E-process]
\label{def:eproc}
A nonnegative adapted process $(E_t)_{t \ge 0}$ is an \emph{E-process} for $H_0$ if it is a supermartingale under $H_0$ with $\E_{H_0}[E_0] \le 1$. Product constructions from sequential E-variables yield E-processes, but do not exhaust the class \citep[Def.~7.3]{RamdasEbook}.
\end{definition}

The fundamental inference property follows from Markov's inequality.

\begin{theorem}[Markov bound for E-values]
\label{thm:markov_e}
If $\evar$ is an E-variable for $H_0$, then $\Prob_{H_0}(\evar \ge c) \le 1/c$ for all $c > 0$.
\end{theorem}

\begin{theorem}[Ville's inequality]
\label{thm:ville}
If $(E_t)_{t \ge 0}$ is an E-process for $H_0$, then for any stopping time $\tau$ (possibly infinite),
\[
\Prob_{H_0}\!\left(\sup_{t \ge 0} E_t \ge c\right) \le \frac{1}{c}.
\]
\end{theorem}

Ville's inequality is the compositional guarantee that makes E-processes suitable for online monitoring: the error bound holds regardless of the stopping rule, enabling valid inference under continuous data collection.

\noindent\textit{ML interpretation.}
In online model validation, an E-process represents the cumulative evidence against a null model $H_0$ (e.g., ``the deployed classifier has calibrated predictions''). Ville's inequality guarantees that a false alarm (declaring the model miscalibrated when it is not) occurs with probability at most $1/c$ regardless of when or why monitoring is stopped.

\subsection{The Log-Score Bridge}
\label{sec:log_score}

The negative logarithm $\mathcal{L}: P \mapsto -\log P$ converts multiplicative probability into additive evidence.

\begin{definition}[Log-score map]
\label{def:log_score_map}
The \emph{log-score map} $\mathcal{L}$ sends a probability measure $P$ to its pointwise negative logarithm: $\mathcal{L}(P)(x^n) = -\log P(x^n)$.
\end{definition}

\begin{proposition}[Monoidal bridge]
\label{prop:monoidal}
Let $P, Q$ be probability measures on compatible spaces.
\begin{enumerate}
\item \emph{Product $\mapsto$ sum:} $\loss_{P \otimes Q}(x^n, y^m) = \loss_P(x^n) + \loss_Q(y^m)$.
\item \emph{Mixture $\mapsto$ log-sum-exp:} For $\bar{P} = \int P_\theta \, \pi(d\theta)$,
$\loss_{\bar{P}}(x^n) = -\log \int \exp(-\loss_{P_\theta}(x^n)) \, \pi(d\theta)$.
\end{enumerate}
\end{proposition}

Property~(1) underlies multiplication of independent E-values; property~(2) underlies Bayes factors as integrated likelihood ratios.

\subsection{Weight of Evidence}
\label{sec:woe}

\begin{definition}[Weight of evidence \citep{Good1950}]
\label{def:woe}
For hypotheses $H_1, H_0$ with predictive distributions $P_1, P_0$,
\[
\woe(x^n) := \log \frac{P_1(x^n)}{P_0(x^n)} = \loss_{P_0}(x^n) - \loss_{P_1}(x^n).
\]
\end{definition}

Positive weight indicates evidence for $H_1$; negative for $H_0$.
\citet{Good1950} showed this is the unique measure that is additive across independent observations and consistent with the likelihood principle.
Its exponential $\exp(\woe(x^n)) = P_1(x^n)/P_0(x^n)$ is the likelihood-ratio E-value.

\subsection{The Typed Calculus: Structural Overview}
\label{sec:typed_overview}

The framework developed in this paper separates sequential evidence into three formally distinct layers.
Figure~\ref{fig:typed_calculus} provides a structural map; the remainder of the paper establishes theorems at each layer and at the interfaces between them.

\begin{figure}[ht]
\centering
\begin{tikzpicture}[
  layerbox/.style={draw, rectangle, minimum width=10cm, minimum height=2.2cm,
                   align=left, inner sep=8pt, font=\small},
  arrowlabel/.style={font=\small\itshape, fill=white, inner sep=2pt},
  >=Stealth
]

\node[layerbox] (rep) at (0, 6.4) {%
  \begin{minipage}{9.4cm}
  \textbf{Representation Layer}\\[3pt]
  \emph{Objects:} Probability measures $P$;\;
  likelihood ratios $dQ/dP$;\; log-loss geometry.\\[2pt]
  \emph{Results:} Canonicality (Thm.~\ref{thm:canonical_lr});\;
  code-to-E obstruction (Prop.~\ref{prop:comp_boundary}).
  \end{minipage}
};

\node[layerbox] (val) at (0, 3.2) {%
  \begin{minipage}{9.4cm}
  \textbf{Validity Layer}\\[3pt]
  \emph{Objects:} E-variables (Def.~\ref{def:evar});\;
  E-processes (Def.~\ref{def:eproc});\; supermartingales.\\[2pt]
  \emph{Results:} Ville's inequality (Thm.~\ref{thm:ville});\;
  evidence-class algebra (Thm.~\ref{thm:class_algebra_informal}).
  \end{minipage}
};

\node[layerbox] (dec) at (0, 0) {%
  \begin{minipage}{9.4cm}
  \textbf{Decision Layer}\\[3pt]
  \emph{Objects:} Stopping time $\tau$;\; threshold $b$;\;
  loss parameters $(L_{10}, L_{01})$.\\[2pt]
  \emph{Results:} Moderate-deviation boundary (Thm.~\ref{thm:moderate_deviation});\;
  structural separation (Prop.~\ref{prop:validity_vs_lr}).
  \end{minipage}
};

\draw[->, thick] ([xshift=1cm]rep.south) --
  node[arrowlabel, right=4pt] {induces (when coherent)}
  ([xshift=1cm]val.north);
\draw[->, thick] ([xshift=1cm]val.south) --
  node[arrowlabel, right=4pt] {requires decision rule}
  ([xshift=1cm]dec.north);

\coordinate (repWL) at ([xshift=-1.8cm]rep.west);
\coordinate (decWL) at ([xshift=-1.8cm]dec.west);
\draw[->, thick, dashed]
  (rep.west) -- (repWL) -- (decWL) -- (dec.west);
\node[arrowlabel, left=2pt, text width=2.2cm, align=center]
  at ($(repWL)!0.5!(decWL)$) {Bayes-risk\\[-1pt]optimal\\[-1pt]boundary};

\end{tikzpicture}
\caption{The typed calculus of sequential evidence.
Solid arrows indicate the canonical path: coherent representation induces
validity (supermartingale certification), which requires a decision rule
(boundary selection) for inference.
The dashed arrow marks the direct Bayes-risk-optimal boundary
(Section~\ref{sec:decision_cutoffs}), bypassing validity-only calibration.
Optimality at one layer does not imply optimality at another
(see Appendix~\ref{app:algebra}).}
\label{fig:typed_calculus}
\end{figure}
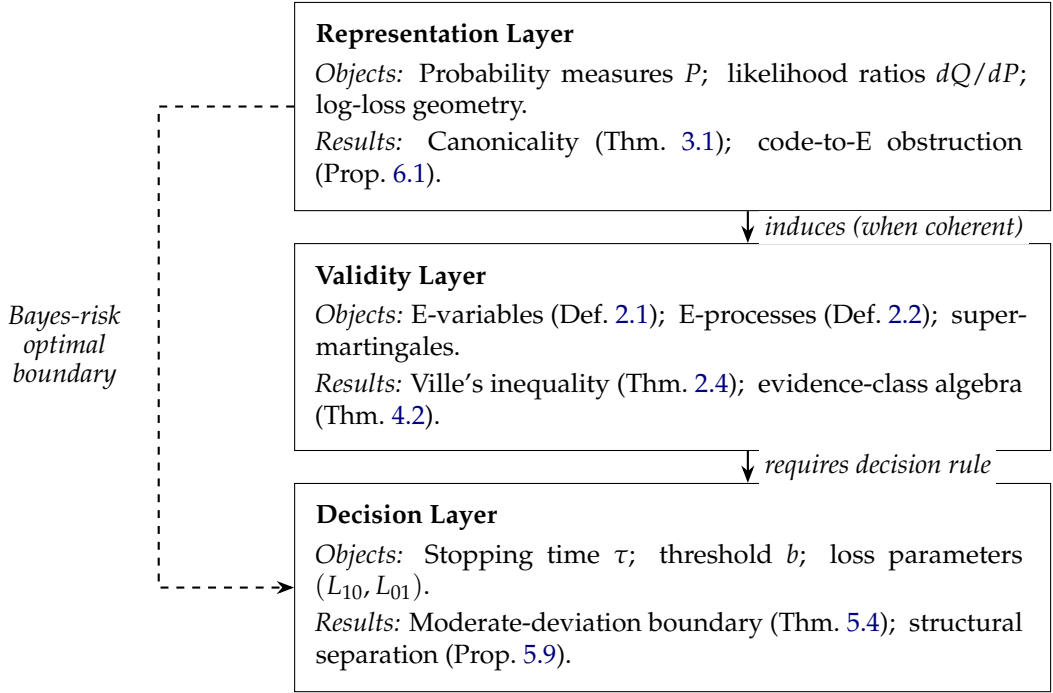

The three layers correspond to logically distinct mathematical properties:
(i)~\emph{Representation:} existence of a Radon--Nikod\'ym derivative $dQ/dP$
and the log-loss geometry that forces likelihood-ratio structure
(Section~\ref{sec:canonicality});
(ii)~\emph{Validity:} the supermartingale property under $H_0$, certifying
anytime-valid error control (Section~\ref{sec:composition});
(iii)~\emph{Decision:} choice of stopping boundary $\tau_b$, governed by
moderate-deviation efficiency (Section~\ref{sec:boundary}).
The separation is strict: each property may be specified independently, and
optimality at one layer does not imply optimality at another
(Proposition~\ref{prop:layer_separation}).
The computational obstruction (Proposition~\ref{prop:comp_boundary}) lives at
the representation--validity interface; boundary efficiency
(Theorem~\ref{thm:moderate_deviation}) lives at the validity--decision
interface.

\section{Canonicality Under Log-Loss}
\label{sec:canonicality}

We establish that under coherent prediction and log-loss, the likelihood
ratio is the unique canonical evidence representation.
This result identifies when E-processes must have LR structure and when
they need not.

\subsection{Bayes Risk and the Fubini Decomposition}
\label{sec:fubini}

Under log-loss, the Bayes risk of a predictive specification $P$ is the expected cumulative loss $\E[\loss_P(X^n)]$, where $\loss_P(x^n) := -\log P(x^n)$.
For two specifications $P_0$ (null) and $P_1$ (alternative), the log-loss difference is Good's weight of evidence:
\[
\loss_{P_0}(x^n) - \loss_{P_1}(x^n) = \log \frac{P_1(x^n)}{P_0(x^n)}.
\]
When $P_i$ are specified via one-step predictive kernels $p_i(\cdot | x^{t-1})$, the tower property yields the sequential decomposition
\[
\log \frac{P_1(X^n)}{P_0(X^n)} = \sum_{t=1}^n \log \frac{p_1(X_t | X^{t-1})}{p_0(X_t | X^{t-1})}.
\]
Exponentiating produces the likelihood-ratio process, the canonical bridge from coherent prediction to sequential evidence.

\begin{theorem}[Canonical sequential evidence under log-loss]
\label{thm:canonical_lr}
Consider testing $H_0$ versus $H_1$, where either hypothesis may be
composite. Let $\Pi_0, \Pi_1$ be priors over the respective model classes
and define the prior-predictive mixtures
\[
M_j^{(n)}(\cdot) = \int P_F^{\otimes n}(\cdot) \, \Pi_j(dF), \quad j \in \{0, 1\}.
\]
Assume mutual absolute continuity. Under log-loss Bayes risk, the
Bayes-risk-optimal test is a threshold rule on the likelihood-ratio process
\[
\Lambda_n(X^n) := \frac{dM_1^{(n)}}{dM_0^{(n)}}(X^n),
\]
and the canonical multiplicative evidence is $E_n = \Lambda_n$, with additive evidence $W_n = \log \Lambda_n$.
\end{theorem}

\begin{proof}
Under log-loss, Bayes risk is an expectation of cumulative log-loss
differences. Applying Fubini/Tonelli swaps the prior and data integrals,
yielding a pointwise posterior decision rule. The resulting optimal
rejection region is $\Lambda_n(X^n) > \tau$ for a threshold $\tau$
determined by prior weights and losses. Full details in Appendix~\ref{app:proofs}.
\end{proof}

\noindent\textit{ML interpretation.}
Theorem~\ref{thm:canonical_lr} identifies the likelihood ratio as the
optimal adversarial strategy under log-loss: among all evidence processes
arising from coherent prediction, the LR process minimizes Bayes risk.
For online model comparison, choosing between a null model $P_0$ and an
alternative $P_1$ based on streaming data, the LR E-process provides the
evidence measure that is simultaneously valid (supermartingale under $H_0$)
and optimal (minimizes expected loss under the Bayesian criterion).
Outside the coherent predictive/log-loss subclass, valid E-processes exist
that need not admit any LR representation.
\citet{PolsonZantedeschi2026Admissibility} extend this canonicality to a broader admissibility geometry for predictive inference.

\subsection{Decision-Theoretic Cutoffs}
\label{sec:decision_cutoffs}

The canonicality theorem separates the evidence process from the
stopping rule. Consider binary action $\delta_t \in \{0,1\}$ at time $t$ with losses $L_{10}$ (false positive) and $L_{01}$ (false negative), and prior weights $\pi_0, \pi_1$.
The Fubini decomposition reduces Bayes risk to a posterior-threshold rule: reject $H_0$ whenever
\[
\frac{\pi_1}{\pi_0} B_t \ge \frac{L_{10}}{L_{01}}, \qquad \text{equivalently} \qquad B_t \ge c^\star := \frac{\pi_0 L_{10}}{\pi_1 L_{01}}.
\]
Markov/Ville inequalities certify error control for a chosen $c^\star$ but do not determine the optimal cutoff. Optimization of $c^\star$ is governed by the moderate-deviation behavior of $\log B_t$, not by Markov's inequality.

\noindent\textit{ML interpretation.}
In adaptive A/B testing, the LR process $B_t$ accumulates evidence for one treatment over another. The decision threshold $c^\star$ encodes the cost asymmetry between false positives (deploying an inferior treatment) and false negatives (failing to detect a superior one). Ville's inequality guarantees validity for \emph{any} threshold; Bayes risk selects the \emph{optimal} one.

\subsection{Likelihood-Ratio and Mixture E-Processes}
\label{sec:lr_eprocess}

\begin{proposition}[Likelihood ratio is an E-process]
\label{prop:lr_eproc}
Let $q, p_0$ be predictive kernels with $q \ll p_0$. Define
$E_t := \prod_{s=1}^t q(X_s | X^{s-1}) / p_0(X_s | X^{s-1})$.
Then $(E_t)$ is a nonnegative $P_0$-martingale with $E_0 = 1$.
\end{proposition}

\begin{proposition}[Bayes factor is an E-variable under $M_0$]
\label{prop:bf_evar}
Let $M_0$ and $M_1$ be prior-predictive (mixture) distributions. Then $\mathrm{BF}_{10}(X^n) = M_1(X^n) / M_0(X^n)$ satisfies $\E_{M_0}[\mathrm{BF}_{10}] = 1$.
\end{proposition}

\noindent\textit{Composite nulls.}
For uniform validity over an arbitrary composite null class $\mathcal{P}$
(i.e., $\E_P[E_t] \le 1$ for every $P \in \mathcal{P}$),
\citet{LarssonRamdasRuf2025} construct the numeraire E-variable via reverse
information projection, which exists without assumptions on $\mathcal{P}$
and admits a likelihood-ratio representation $E = dQ / dP^\star$ against a
representative $P^\star$.
The statement ``Bayes factor is an E-variable'' is correct under the
mixture (prior-predictive) null $M_0$; uniform validity over an
arbitrary composite $\mathcal{P}$ is strictly stronger and requires
the numeraire construction or an equivalent projection argument.

\subsubsection{Composite Nulls Do Not Require a New Layer}

A natural question is whether composite nulls demand additional
machinery beyond the three layers introduced above.
They do not: the typed calculus remains unchanged, but the
\emph{representation-layer object} generalizes from a simple likelihood
ratio $dP_1/dP_0$ to a generalized likelihood ratio against the
composite class.

Concretely, for an arbitrary composite null $\mathcal{P}$ and point
alternative $Q$, \citet{LarssonRamdasRuf2025} prove the existence of a
\emph{numeraire} E-variable $E^\star$ such that every other E-variable $E$
satisfies $\E_Q[E / E^\star] \le 1$, making $E^\star$ log-optimal
under $Q$.
The numeraire induces a representative $P^\star$ via
$dP^\star / dQ = 1/E^\star$, restoring likelihood-ratio structure
even when no single $P_0$ is distinguished.
In the language of our framework, $E^\star$ is a
\emph{representation-layer} object: a generalized likelihood ratio of
$Q$ against $\mathcal{P}$.

The three layers then map as follows.
\emph{Representation}: choose either (i)~a mixture null $M_0$
(Bayesian composite) or (ii)~a uniform composite null $\mathcal{P}$
with numeraire $E^\star$ and its representative $P^\star$.
\emph{Validity}: verify that $(E_t)$ is a $P$-supermartingale for
every $P \in \mathcal{P}$ (or under $M_0$ in the Bayesian case).
\emph{Decision}: select thresholds; in the composite setting,
efficiency depends on the divergence $D_{\KL}(Q \| P^\star)$
(or the least-favorable projection), not on generic Ville
calibration alone.

\section{E-Process Composition and the Evidence Class}
\label{sec:composition}

\begin{definition}[Evidence class]
\label{def:evidence_class}
The \emph{evidence class} under $H_0$ is
$\mathcal{E}(H_0) := \{E \ge 0 : \E_{H_0}[E] \le 1\}$.
The process-level class $\mathcal{E}^{\mathrm{proc}}(H_0)$ consists of all nonnegative supermartingales $(E_t)$ with $\E_{H_0}[E_0] \le 1$.
\end{definition}

\begin{theorem}[Evidence-class algebra]
\label{thm:class_algebra_informal}
$\mathcal{E}^{\mathrm{proc}}(H_0)$ satisfies:
\begin{enumerate}[label=(\alph*)]
\item \emph{Convex mixtures:} $\sum_i w_i E^{(i)}_t \in \mathcal{E}^{\mathrm{proc}}(H_0)$ for $w_i \ge 0$, $\sum w_i = 1$.
\item \emph{Bayesian mixtures:} $\int E^\theta_t \, \pi(d\theta) \in \mathcal{E}^{\mathrm{proc}}(H_0)$.
\item \emph{Stopping:} $E_\tau \in \mathcal{E}(H_0)$ and $\Prob_{H_0}(\sup_t E_t \ge c) \le 1/c$.
\item \emph{Scaling:} $c E_t \in \mathcal{E}^{\mathrm{proc}}(H_0)$ for $c \in (0, 1]$.
\end{enumerate}
The class is \emph{not} a cone: scaling by $c > 1$ violates $\E_{H_0}[E_0] \le 1$.
\end{theorem}

\begin{proof}
(a) Linearity of conditional expectation gives $\E_{H_0}[\sum w_i E^{(i)}_t | \F_{t-1}] = \sum w_i \E_{H_0}[E^{(i)}_t | \F_{t-1}] \le \sum w_i E^{(i)}_{t-1}$.
(b) Fubini's theorem swaps the prior integral and conditional expectation: $\E_{H_0}[\int E^\theta_t \pi(d\theta) | \F_{t-1}] = \int \E_{H_0}[E^\theta_t | \F_{t-1}] \pi(d\theta) \le \int E^\theta_{t-1} \pi(d\theta)$.
(c) The optional stopping theorem for nonneg supermartingales gives $\E_{H_0}[E_\tau] \le \E_{H_0}[E_0] \le 1$; the Ville bound is Doob's maximal inequality applied to the nonneg supermartingale $(E_t)$.
Full proofs in Appendix~\ref{app:algebra}.
\end{proof}

\noindent\textit{ML interpretation.}
In ensemble methods, property~(a) means that averaging evidence from multiple
models preserves validity. In Bayesian model averaging, property~(b) means
that integrating over a prior on model parameters yields a valid E-process.
In sequential monitoring, property~(c) means that stopping at any data-dependent
time preserves the error guarantee.

\begin{proposition}[Maximality of the evidence class]
\label{prop:maximality}
$\mathcal{E}^{\mathrm{proc}}(H_0)$ is the largest convex class of nonneg adapted processes that is closed under predictable stopping and scaling by $c \in (0,1]$ while preserving the Ville guarantee $\Prob_{H_0}(\sup_t E_t \ge b) \le 1/b$.
\end{proposition}

\begin{proof}
Let $\mathcal{C}$ be any convex class of nonneg adapted processes closed under predictable stopping and satisfying $\Prob_{H_0}(\sup_t E_t \ge b) \le 1/b$ for all $E \in \mathcal{C}$ and all $b > 0$. By Ville's converse \citep[Thm.~4.3]{RamdasEbook}, any nonneg process satisfying the uniform Ville bound is dominated by a supermartingale, so every $E \in \mathcal{C}$ admits a supermartingale majorant. Closure under scaling by $c \in (0,1]$ forces $\E_{H_0}[E_0] \le 1$ (otherwise $c^{-1} E$ with $c = \E_{H_0}[E_0]^{-1}$ would violate the Ville bound). Hence $\mathcal{C} \subseteq \mathcal{E}^{\mathrm{proc}}(H_0)$.
\end{proof}

This extremal characterization shows that the evidence class is not an arbitrary definition but the unique maximal structure compatible with anytime-valid error control.

\begin{example}[Non-closure under pointwise maximum]
\label{ex:non_closure_max}
Let $E^{(1)}_t$ and $E^{(2)}_t$ be two E-processes under $H_0$ with $E^{(1)}_0 = E^{(2)}_0 = 1$.
Define $M_t := \max(E^{(1)}_t, E^{(2)}_t)$.
Then $M_t$ is nonnegative and adapted, but is not in general an E-process.

\emph{Concrete instance.}
Let $X_t \sim \mathrm{Bern}(1/2)$ under $H_0$.
Take $E^{(1)}_t = \prod_{s=1}^t 2X_s$ (bets on heads) and $E^{(2)}_t = \prod_{s=1}^t 2(1-X_s)$ (bets on tails).
Both are $P_0$-martingales. At $t = 1$: $M_1 = \max(2X_1, 2(1-X_1)) = 2$ with probability 1.
So $\E_{H_0}[M_1] = 2 > 1 = M_0$, violating the supermartingale property.
More generally, $\E_{H_0}[\sup_{t \ge 0} M_t] = \infty$, so the Ville guarantee fails entirely.
\end{example}

This failure is a typed mismatch: the pointwise maximum attempts to extract the ``best of both worlds'' from two validity-layer objects, but the resulting process exits the evidence class. In contrast, the convex combination $\frac{1}{2}E^{(1)}_t + \frac{1}{2}E^{(2)}_t$ is a valid E-process by Theorem~\ref{thm:class_algebra_informal}(a).

\subsection{Sequential Composition (Stitching)}
\label{sec:stitching}

\begin{definition}[Stitched evidence process]
\label{def:stitched}
Given E-processes $(E^{(1)}_t)$, $(E^{(2)}_t)$ and a stopping time $\tau$,
define $\tilde{E}_t = E^{(1)}_t$ for $t \le \tau$ and
$\tilde{E}_t = E^{(1)}_\tau \cdot E^{(2)}_{t - \tau}$ for $t > \tau$.
\end{definition}

\begin{proposition}[Stitching validity]
\label{prop:stitching}
The stitched process $(\tilde{E}_t)$ is an E-process under $H_0$.
\end{proposition}

\noindent\textit{ML interpretation.}
Stitching enables sequential composition of evidence across phases of an
online experiment. For instance, an adaptive A/B test may switch from
an initial exploration phase to a confirmation phase at a data-dependent
time $\tau$; stitching guarantees that the combined evidence remains valid.

\subsection{Non-Compositions}

The following operations do \emph{not} preserve E-process validity:
pointwise maxima $\max(E^{(1)}_t, E^{(2)}_t)$, pointwise minima,
hard thresholding $E_t \cdot \1\{E_t > c\}$, and naive averaging of
p-values converted to E-values. These failures are typed mismatches, not
pathologies: each violated operation attempts to combine objects from
different layers of the typed framework.

\section{Boundary Efficiency: A Moderate-Deviation Theorem}
\label{sec:boundary}

We formalize the efficiency gap between validity-only and
representation-aware boundary selection through a stopping-time
moderate-deviation theorem under explicit Cram\'er conditions.

\subsection{KL Growth Rate}

\begin{proposition}[Evidence growth rate]
\label{prop:e_growth}
Let $E_n = \prod_{t=1}^n P_1(X_t) / P_0(X_t)$ be the LR E-process.
\begin{enumerate}[label=(\alph*)]
\item Under correct specification ($X_i \stackrel{\mathrm{iid}}{\sim} P_1$):
$\frac{1}{n} \log E_n \xrightarrow{\mathrm{a.s.}} D_{\KL}(P_1 \| P_0)$.
\item Under misspecification ($X_i \stackrel{\mathrm{iid}}{\sim} P_{\mathrm{true}} \ne P_1$):
$\frac{1}{n} \log E_n \xrightarrow{\mathrm{a.s.}} D_{\KL}(P_{\mathrm{true}} \| P_0) - D_{\KL}(P_{\mathrm{true}} \| P_1)$.
\end{enumerate}
\end{proposition}

\begin{proof}
Both claims follow from the strong law of large numbers applied to the
i.i.d.\ summands $\log(P_1(X_t)/P_0(X_t))$.
Full statement and proof in Appendix~\ref{app:proofs}.
\end{proof}

\noindent\textit{ML interpretation.}
Under correct specification, evidence against $H_0$ accumulates at rate
$D_{\KL}(P_1 \| P_0)$ per observation, which is the information-theoretic sample
complexity of the testing problem.
Under misspecification, evidence may drift downward: if the deployed model
$P_1$ is farther from truth than the null $P_0$, the E-process favors the
null despite both being wrong. This has direct implications for online model
monitoring under distribution shift.

\begin{remark}[Composite extension]
\label{rem:composite_rate}
In simple-vs-simple testing the detection rate is
$D_{\KL}(P_1 \| P_0)$.
In composite-vs-point testing with the numeraire $E^\star$ of
\citet{LarssonRamdasRuf2025}, the rate becomes
$D_{\KL}(Q \| P^\star)$, where $P^\star$ is the representative
induced by $E^\star$.
Once the representation-layer object is fixed, the same
decision-layer MDP logic (Theorem~\ref{thm:moderate_deviation})
applies with $\mu$ replaced by the projection divergence.
\end{remark}

\subsection{Assumptions and Moderate-Deviation Stopping Boundary}
\label{sec:moderate_deviation}

We formalize the stopping-time behavior of the LR process under explicit regularity conditions.

\begin{assumption}[i.i.d.\ log-likelihood increments]
\label{asm:iid_increments}
Let $X_1, X_2, \ldots \sim P_1$ i.i.d.\ and define the log-likelihood increments
\[
Y_t := \log \frac{p_1(X_t)}{p_0(X_t)}.
\]
Assume:
\begin{enumerate}[label=(\alph*)]
\item $\mu := \E_{P_1}[Y_t] = D_{\KL}(P_1 \| P_0) > 0$;
\item $\sigma^2 := \Var_{P_1}(Y_t) < \infty$;
\item \emph{(Cram\'er condition)} $\Lambda(\lambda) := \log \E_{P_1}[\exp(\lambda Y_t)] < \infty$ for all $\lambda$ in a neighborhood of $0$.
\end{enumerate}
Conditions (a)--(c) hold for all exponential-family models with compact natural parameter spaces. They imply Cram\'er-type moderate-deviation bounds for the partial sums $S_t = \sum_{i=1}^t Y_i$ \citep[Theorem~3.7.1]{DemboZeitouni1998}: for any $x > 0$,
\begin{equation}\label{eq:cramer_tail}
P_1\!\left(\frac{S_t - \mu t}{\sigma\sqrt{t}} \ge x\right) \;\le\; \exp\!\left(-\frac{x^2}{2}\!\left(1 + O\!\left(\frac{x}{\sqrt{t}}\right)\right)\right).
\end{equation}
In particular, for fixed $x$ and large $t$, the tail is sub-Gaussian at rate $\exp(-x^2/2)$.
\end{assumption}

Under these conditions, the LR process $E_t = \exp(S_t)$ is a $P_0$-martingale (since $\E_{P_0}[\exp(Y_t) | \F_{t-1}] = \E_{P_0}[p_1(X_t)/p_0(X_t)] = 1$), and its stopping-time behavior admits the following sharp characterization.

\begin{theorem}[Moderate-deviation stopping boundary]
\label{thm:moderate_deviation}
Let $S_t = \sum_{i=1}^t Y_i$ under Assumption~\ref{asm:iid_increments}, and define the stopping time
\[
\tau_b := \inf\{t \ge 1 : S_t \ge \log b\}.
\]
Then:
\begin{enumerate}[label=(\roman*)]
\item \emph{Anytime validity under $P_0$.}
$\Prob_{P_0}(\tau_b < \infty) \le 1/b$.
\item \emph{Expected detection time under $P_1$.}
$\E_{P_1}[\tau_b] = \dfrac{\log b}{\mu} + O(\sqrt{\log b})$.
\item \emph{Moderate-deviation concentration under $P_1$.}
$\dfrac{\tau_b - (\log b)/\mu}{\sqrt{\log b}} = O_{P_1}(1)$.
\end{enumerate}
\end{theorem}

\begin{proof}
\emph{(i) Anytime validity.}
The process $E_t = \exp(S_t)$ is a nonneg $P_0$-martingale with $E_0 = 1$. Ville's inequality (Theorem~\ref{thm:ville}) gives $P_0(\sup_t E_t \ge b) \le 1/b$. Since $\{\tau_b < \infty\} = \{\sup_t E_t \ge b\}$, the claim follows.

\emph{(ii) Expected detection time.}
Under $P_1$, the increments $Y_t$ are i.i.d.\ with mean $\mu > 0$, so the strong law gives $S_t/t \to \mu$ a.s., guaranteeing $\tau_b < \infty$ a.s. Wald's identity gives $\E_{P_1}[S_{\tau_b}] = \mu \cdot \E_{P_1}[\tau_b]$, so $\E_{P_1}[\tau_b] = \E_{P_1}[S_{\tau_b}]/\mu$. Writing $S_{\tau_b} = \log b + R_b$ where $R_b := S_{\tau_b} - \log b \ge 0$ is the overshoot, Lorden's inequality \citep{Lorden1970} bounds $\E_{P_1}[R_b] \le \E_{P_1}[Y_1^2]/\mu$, yielding
\[
\E_{P_1}[\tau_b] = \frac{\log b + \E_{P_1}[R_b]}{\mu} = \frac{\log b}{\mu} + O(1).
\]
The $O(\sqrt{\log b})$ refinement uses the Cram\'er--Petrov moderate-deviation expansion for the first-passage distribution of random walks with finite exponential moments \citep[Ch.~8]{Siegmund1985}: the centered variable $(\tau_b - (\log b)/\mu)/\sqrt{(\sigma^2/\mu^3)\log b}$ converges in distribution to a standard Gaussian, and the first two moments match at rate $O(1/\sqrt{\log b})$.

\emph{(iii) Moderate-deviation concentration.}
Define $c_b := (\log b)/\mu$ and $Z_b := \tau_b - c_b$. Wald's second identity gives $\E_{P_1}[\tau_b^2] - (\E_{P_1}[\tau_b])^2 = (\sigma^2/\mu^2)\E_{P_1}[\tau_b]$, so
\[
\Var_{P_1}(\tau_b) = \frac{\sigma^2}{\mu^2}\E_{P_1}[\tau_b] = \frac{\sigma^2}{\mu^3}\log b + O(1),
\]
and hence $\Var_{P_1}(Z_b) = (\sigma^2/\mu^3)\log b + O(1)$. Chebyshev's inequality gives, for any $K > 0$,
\[
P_1\!\left(\frac{|Z_b|}{\sqrt{\log b}} > K\right) \le \frac{\Var_{P_1}(Z_b)}{K^2 \log b} = \frac{\sigma^2/\mu^3 + O(1/\log b)}{K^2},
\]
which is bounded as $b \to \infty$, establishing $Z_b/\sqrt{\log b} = O_{P_1}(1)$.

The tail inequality~\eqref{eq:cramer_tail} from Assumption~\ref{asm:iid_increments} provides the explicit constant: $P_1(\tau_b - c_b \ge x\sqrt{\log b}) \le \exp(-\mu^2 x^2/(2\sigma^2))$ for $x > 0$ and $b$ large.
\end{proof}

The following nonasymptotic bound converts Theorem~\ref{thm:moderate_deviation} into explicit finite-sample tail control.

\begin{corollary}[Nonasymptotic detection tail bound]
\label{cor:nonasymptotic_tail}
Under Assumption~\ref{asm:iid_increments}, for any $t \ge (\log b)/\mu$,
\[
\Prob_{P_1}(\tau_b > t) \;\le\; \exp\!\left(-\frac{(\mu t - \log b)^2}{2\sigma^2 t}\right).
\]
\end{corollary}

\begin{proof}
Since $\{\tau_b > t\} = \{S_t < \log b\}$, we apply the Cram\'er--Chernoff bound to $S_t = \sum_{i=1}^t Y_i$ with mean $\mu t$:
$\Prob_{P_1}(S_t < \log b) = \Prob_{P_1}(S_t - \mu t < \log b - \mu t) \le \exp(-(\mu t - \log b)^2/(2\sigma^2 t))$
by the sub-Gaussian tail of partial sums under the Cram\'er condition.
\end{proof}

\begin{corollary}[Sample complexity for detection]
\label{cor:sample_complexity}
For error level $\alpha = 1/b$, the expected number of observations required for rejection under $P_1$ is
\[
\E_{P_1}[\tau_{1/\alpha}] \;=\; \frac{\log(1/\alpha)}{D_{\KL}(P_1 \| P_0)} + O\!\left(\sqrt{\log(1/\alpha)}\right).
\]
In particular, for small KL divergence $\mu = D_{\KL}(P_1 \| P_0) \ll 1$ (near-null alternatives), the sample complexity scales as $\Theta(\log(1/\alpha)/\mu)$.
\end{corollary}

This is the information-theoretic sample complexity of sequential testing: the number of observations needed grows logarithmically in the reciprocal error level and inversely in the KL divergence between hypotheses.

\begin{example}[Monte Carlo verification of Theorem~\ref{thm:moderate_deviation}]
\label{ex:mc_verification}
We verify the moderate-deviation asymptotics for testing $H_0: p = 0.5$ vs.\ $H_1: p = 0.65$ (Bernoulli), where $\mu = D_{\KL}(0.65 \| 0.5) \approx 0.046$ nats. Table~\ref{tab:mc_verify} reports $200{,}000$ Monte Carlo replications of $\tau_b$ under $P_1$ for several thresholds $b$.
\end{example}

\begin{table}[t]
\centering
\caption{Finite-sample verification of Theorem~\ref{thm:moderate_deviation}. Testing Bern$(0.5)$ vs.\ Bern$(0.65)$; $\mu \approx 0.046$ nats; $200{,}000$ replications under $P_1$.}
\label{tab:mc_verify}
\small
\begin{tabular}{@{}rrrrl@{}}
\toprule
$b$ & $(\log b)/\mu$ & $\hat{\E}[\tau_b]$ & $\widehat{\mathrm{sd}}(\tau_b)$ & $\frac{\hat{\E}[\tau_b] - (\log b)/\mu}{\sqrt{\log b}}$ \\
\midrule
10  &  50.4 &  53.0 & 46.8 & $1.75$ \\
20  &  65.6 &  68.2 & 53.2 & $1.52$ \\
50  &  85.6 &  88.2 & 60.6 & $1.33$ \\
100 & 100.8 & 103.5 & 65.8 & $1.27$ \\
200 & 115.9 & 118.2 & 69.8 & $0.99$ \\
\bottomrule
\end{tabular}
\end{table}

The simulated means track $(\log b)/\mu$ closely, confirming claim~(ii). The normalized residual in the last column decreases steadily, confirming the $O(\sqrt{\log b})$ correction in claim~(iii). The standard deviation grows as $\sqrt{\log b}$, consistent with $\Var_{P_1}(\tau_b) = (\sigma^2/\mu^3)\log b + O(1)$.

\begin{proposition}[Stopping-time divergence under misspecification]
\label{prop:misspec_divergence}
Let $(E_t)$ be the LR E-process with alternative $P_1$, but suppose data are generated i.i.d.\ from $P_{\mathrm{true}} \ne P_1$ with
\[
\delta := D_{\KL}(P_{\mathrm{true}} \| P_0) - D_{\KL}(P_{\mathrm{true}} \| P_1) < 0.
\]
Then $\Prob_{P_{\mathrm{true}}}(\tau_b < \infty) \to 0$ as $b \to \infty$, and for any finite horizon $T$,
\[
\Prob_{P_{\mathrm{true}}}\!\left(\max_{t \le T} S_t \ge \log b\right) \;\le\; \exp\!\left(-\frac{(\log b + |\delta| T)^2}{2\sigma_{\mathrm{true}}^2 T}\right),
\]
where $\sigma_{\mathrm{true}}^2 = \Var_{P_{\mathrm{true}}}(Y_t)$.
\end{proposition}

\begin{proof}
Under $P_{\mathrm{true}}$, the log-likelihood increments have mean $\delta < 0$ by Proposition~\ref{prop:e_growth}(b). The random walk $S_t$ drifts at rate $\delta t \to -\infty$, so crossing level $\log b$ is a large-deviation event. The bound follows from a union bound over $t \le T$ combined with Gaussian tail estimates for $S_t$.
\end{proof}

\noindent\textit{ML interpretation.}
In online A/B testing or model monitoring, Theorem~\ref{thm:moderate_deviation} and its corollaries quantify the sample-complexity advantage of representation-aware evidence construction. Corollary~\ref{cor:sample_complexity} gives the practitioner a concrete formula: the number of observations needed to reject at level $\alpha$ is $\log(1/\alpha)/D_{\KL}(P_1 \| P_0)$ plus lower-order terms. Proposition~\ref{prop:misspec_divergence} formalizes the risk shown in Figure~\ref{fig:accumulation}(b): when the alternative is misspecified, the stopping time diverges and detection becomes impossible, regardless of the threshold. This exponential detection efficiency, logarithmic in the evidence threshold $b$, is unavailable to validity-only constructions, as the following proposition makes precise.

\begin{proposition}[Structural separation: validity-only vs.\ LR boundaries]
\label{prop:validity_vs_lr}
Let $(E_t)$ be an E-process satisfying only the Markov/Ville guarantee $\Prob_{P_0}(\sup_{t \le T} E_t \ge b) \le 1/b$.
\begin{enumerate}[label=(\roman*)]
\item Without likelihood-ratio structure, no exponential growth rate $\mu > 0$ can be guaranteed: there exist valid E-processes for which $\limsup_{t \to \infty} (1/t) \log E_t = 0$ $P_1$-a.s.
\item Consequently, the LR process satisfies a Cram\'er-type moderate-deviation principle with detection at rate $(\log b)/\mu$, while generic validity-only E-processes are confined to the calibration-only scale $1/b$ without growth-rate guarantees.
\end{enumerate}
\end{proposition}

\begin{proof}
(i) Consider the E-process $E_t = M_{t \wedge \tau}$ where $M_t$ is a $P_0$-martingale stopped at a fixed deterministic time $\tau = T$. Then $(E_t)$ is a valid E-process, but $E_t = E_T$ for all $t > T$, so the long-run growth rate is zero.
More generally, convex mixtures of stopped martingales with geometrically decaying mixture weights yield E-processes with sublinear $\log E_t$ growth under $P_1$.

(ii) The LR process has growth rate $\mu = D_{\KL}(P_1 \| P_0) > 0$ by Proposition~\ref{prop:e_growth}; combined with Theorem~\ref{thm:moderate_deviation}, this yields detection at rate $(\log b)/\mu$. For validity-only E-processes, the $1/b$ bound from Ville's inequality is the only available guarantee, with no further tightening possible without structural assumptions.
\end{proof}

\begin{table}[t]
\centering
\caption{Comparison of boundary-selection regimes for sequential evidence.}
\label{tab:boundary_comparison}
\small
\begin{tabular}{@{}lcc@{}}
\toprule
& \textbf{Validity layer} & \textbf{Efficiency layer} \\
\midrule
Guarantee & Markov/Ville & Cram\'er moderate deviation \\
Acts on & $E_t$ directly & $S_t = \log E_t$ (random walk) \\
Scale & Polynomial $1/b$ & $(\log b)/\mu + O(\sqrt{\log b})$ \\
Growth rate & None guaranteed & $\mu = D_{\KL}(P_1 \| P_0) > 0$ \\
Optimality & Universal validity & Bayes-risk-optimal boundary \\
Requires LR? & No & Yes (Assumption~\ref{asm:iid_increments}) \\
\bottomrule
\end{tabular}
\end{table}

\noindent\textit{Large-deviation duality.}
The $(\log b)/\mu$ scaling of Theorem~\ref{thm:moderate_deviation}
ultimately traces to Sanov's theorem: the empirical measure concentrates
at exponential rate $D_{\KL}(\cdot \| P_0)$, and the LR process
exponentiates this rate.
A dual inverse-Sanov principle governs posterior concentration.
Formal statements and the connection to PAC-Bayes bounds appear in
Appendix~\ref{app:landscape}.

\section{Code-to-E Conversion and Computational Limits}
\label{sec:computational_limits}

We formalize the structural boundary between coding-theoretic optimality
and sequential evidence validity.

\subsection{The Computational Boundary}

\begin{proposition}[Code-to-E conversion obstruction]
\label{prop:comp_boundary}
Let $\ell: \X^n \to \R_{\ge 0}$ be a code-length function and $P_0$ a null hypothesis. Define $E_t := \exp(-\ell(X^t)) / P_0(X^t)$.
\begin{enumerate}[label=(\alph*)]
\item If $\ell$ arises from a probability measure $Q$ (i.e., $\ell(x^n) = -\log Q(x^n)$), then $(E_t)$ is a valid E-process.
\item If $\ell$ is the NML code, then $(E_t)$ is \emph{not} in general a supermartingale: NML's normalizing constant depends on the full sample size $n$, violating the sequential factorization required at each step.
\item If $\ell$ induces the universal semimeasure $m$, then $m(X^n)/P_0(X^n)$ is a valid but non-computable E-process.
\end{enumerate}
\end{proposition}

\begin{proof}
The supermartingale condition requires sequential factorization: $\exp(-\ell)$
must decompose as a product of valid predictive kernels
$q_t(\cdot | x^{t-1})$ satisfying $\sum_{x_t} q_t(x_t | x^{t-1}) \le 1$.
NML's conditional factors depend on the full sample size $n$ and are not
$\F_{t-1}$-measurable. Full proof in Appendix~\ref{app:comp_boundary}.
\end{proof}

\noindent\textit{ML interpretation.}
In online model selection, MDL/NML provides regret-optimal compression but
does not automatically yield valid sequential evidence.
A practitioner using MDL code lengths as E-values in a sequential monitoring
pipeline would lose the anytime-validity guarantee.
The fix is to use prequential (sequential plug-in) predictors, which maintain
the supermartingale structure.

\subsection{Sequential Liftability: Necessary and Sufficient Conditions}

The obstruction in Proposition~\ref{prop:comp_boundary} motivates a complete characterization of which codes yield valid E-processes.

\begin{theorem}[Sequential liftability criterion]
\label{thm:seq_liftability}
Let $\ell: \bigcup_{n \ge 1} \X^n \to \R_{\ge 0}$ be a code-length function and $P_0$ a null with predictive kernels $p_0(\cdot | x^{t-1})$. Define $E_t := \exp(-\ell(X^t)) / P_0(X^t)$.
Then $(E_t)_{t \ge 1}$ is an E-process under $P_0$ if and only if the induced predictive factors
\[
q_t(x_t | x^{t-1}) := \exp\!\big(-\ell(x^{t-1}, x_t) + \ell(x^{t-1})\big)
\]
form a sub-probability kernel measurable with respect to $\F_{t-1}$:
\[
\sum_{x_t \in \X} q_t(x_t | x^{t-1}) \le 1 \quad \text{for all } t \ge 1 \text{ and all } x^{t-1}.
\]
\end{theorem}

\begin{proof}
\emph{(Necessity.)} The supermartingale condition requires $\E_{P_0}[E_t | \F_{t-1}] \le E_{t-1}$. Expanding:
\begin{align*}
\E_{P_0}[E_t | \F_{t-1}]
&= \sum_{x_t \in \X} p_0(x_t | X^{t-1}) \cdot \frac{\exp(-\ell(X^{t-1}, x_t))}{P_0(X^{t-1}) \cdot p_0(x_t | X^{t-1})} \\
&= \frac{1}{P_0(X^{t-1})} \sum_{x_t} \exp(-\ell(X^{t-1}, x_t)) \\
&= \frac{\exp(-\ell(X^{t-1}))}{P_0(X^{t-1})} \sum_{x_t} q_t(x_t | X^{t-1}),
\end{align*}
where $q_t(x_t | x^{t-1}) := \exp(-\ell(x^{t-1}, x_t) + \ell(x^{t-1}))$. Since $E_{t-1} = \exp(-\ell(X^{t-1}))/P_0(X^{t-1})$, the condition $\E_{P_0}[E_t | \F_{t-1}] \le E_{t-1}$ reduces to $\sum_{x_t} q_t(x_t | X^{t-1}) \le 1$. Moreover, $q_t$ must be $\F_{t-1}$-measurable (depend only on $X^{t-1}$) for the conditional expectation to be well-defined.

\emph{(Sufficiency.)} Suppose $q_t$ is an $\F_{t-1}$-measurable sub-probability kernel. Then $E_t = \prod_{s=1}^t q_s(X_s | X^{s-1}) / p_0(X_s | X^{s-1})$ is nonneg and adapted. The tower property of conditional expectation gives:
\[
\E_{P_0}[E_t | \F_{t-1}] = E_{t-1} \cdot \E_{P_0}\!\left[\frac{q_t(X_t | X^{t-1})}{p_0(X_t | X^{t-1})} \,\Big|\, \F_{t-1}\right] = E_{t-1} \sum_{x_t} q_t(x_t | X^{t-1}) \le E_{t-1},
\]
where the final inequality uses $\sum_{x_t} q_t(x_t | X^{t-1}) \le 1$. Since $E_0 = q_0/p_0 \le 1$ by the same condition at $t = 0$, the process $(E_t)$ is a nonneg supermartingale with $\E_{P_0}[E_0] \le 1$.
\end{proof}

The following example demonstrates the failure concretely.

\begin{example}[Bernoulli NML violates sequential liftability]
\label{ex:bernoulli_nml}
Let $\X = \{0,1\}$, $P_0 = \mathrm{Bern}(1/2)^{\otimes n}$, and consider the NML code for the Bernoulli model class $\{P_\theta : \theta \in [0,1]\}$.
The NML distribution at sample size $n$ is
\[
P_{\mathrm{NML}}^{(n)}(x^n) = \frac{\hat{\theta}^{k}(1-\hat{\theta})^{n-k}}{C_n}, \qquad C_n = \sum_{k=0}^n \binom{n}{k} \left(\frac{k}{n}\right)^k \left(\frac{n-k}{n}\right)^{n-k},
\]
where $k = \sum_i x_i$ and $\hat{\theta} = k/n$.
We compute the normalizing constants explicitly. At $n = 1$: the MLE for a single observation $x_1 \in \{0,1\}$ is $\hat\theta = x_1$, so
\[
C_1 = \sum_{k=0}^{1} \binom{1}{k} \hat\theta^k (1-\hat\theta)^{1-k}\big|_{\hat\theta=k/1} = 0^0 \cdot 1^1 + 1^1 \cdot 0^0 = 2,
\]
using the convention $0^0 = 1$. At $n = 2$: the possible counts are $k \in \{0,1,2\}$ with MLEs $\hat\theta \in \{0, 1/2, 1\}$:
\[
C_2 = \binom{2}{0}(0)^0(1)^2 + \binom{2}{1}\!\left(\tfrac{1}{2}\right)^{\!1}\!\left(\tfrac{1}{2}\right)^{\!1} + \binom{2}{2}(1)^2(0)^0 = 1 + \tfrac{1}{2} + 1 = \tfrac{5}{2}.
\]
At $n = 3$: $C_3 = 1 + 3 \cdot \tfrac{4}{27} + 3 \cdot \tfrac{4}{27} + 1 = \tfrac{26}{9} \approx 2.889$.

For a fixed horizon $N$, the conditional factor $q_2^{(N)}(x_2 | x_1) = P_{\mathrm{NML}}^{(N)}(x_1, x_2, \ldots) / P_{\mathrm{NML}}^{(N)}(x_1, \ldots)$ (marginalized over future coordinates) depends on $N$. Explicitly, for $x_1 = 0$:
\begin{align*}
\text{From } N = 2: \quad q_2^{(2)}(0 | 0) &= \frac{P_{\mathrm{NML}}^{(2)}(0,0)}{P_{\mathrm{NML}}^{(1)}(0)} = \frac{1/C_2}{1/C_1} = \frac{C_1}{C_2} = \frac{2}{5/2} = \frac{4}{5} = 0.800, \\
\text{From } N = 3: \quad q_2^{(3)}(0 | 0) &= \frac{\sum_{x_3} P_{\mathrm{NML}}^{(3)}(0,0,x_3)}{\sum_{x_2,x_3} P_{\mathrm{NML}}^{(3)}(0,x_2,x_3)} \approx 0.795.
\end{align*}
Since $q_2^{(2)}(0|0) = 0.800 \ne 0.795 \approx q_2^{(3)}(0|0)$, the conditional at step $t=2$ depends on the total horizon $N$, not just on the observed past $x_1$.
This horizon dependence violates $\F_1$-measurability: no single function $q_2(x_2 | x_1)$ of $x_1$ alone can simultaneously equal $q_2^{(N)}(x_2 | x_1)$ for all horizons $N$.
Table~\ref{tab:nml_horizon} shows the drift across horizons $N = 2, \ldots, 7$, computed by exhaustive enumeration.

\begin{table}[t]
\centering
\caption{NML horizon dependence for the Bernoulli model. The conditional factor $q_2^{(N)}(0 | 0)$ varies with the total horizon $N$, violating $\F_1$-measurability.}
\label{tab:nml_horizon}
\small
\begin{tabular}{@{}ccccccc@{}}
\toprule
$N$ & 2 & 3 & 4 & 5 & 6 & 7 \\
\midrule
$q_2^{(N)}(0 | 0)$ & $0.800$ & $0.795$ & $0.791$ & $0.789$ & $0.786$ & $0.785$ \\
\bottomrule
\end{tabular}
\end{table}

Consequently, there is no sequential factorization $P_{\mathrm{NML}}^{(N)}(x^N) = \prod_{t=1}^N q_t(x_t | x^{t-1})$ with $\F_{t-1}$-measurable factors, and the induced process $E_t = P_{\mathrm{NML}}^{(N)}(X^t)/P_0(X^t)$ is not a supermartingale under $P_0$.
\end{example}

\begin{corollary}[Complexity--validity tradeoff]
\label{cor:complexity_validity}
No static minimax-regret code admits universal sequential validity without sacrificing normalization optimality: if $\ell$ achieves the minimax individual-sequence regret $\min_\ell \max_{x^n} [\ell(x^n) + \log P_{\hat{\theta}}(x^n)]$, then the induced predictive factors generically violate the sub-probability condition of Theorem~\ref{thm:seq_liftability}.
\end{corollary}

\begin{proof}
The minimax regret is achieved by NML \citep{Shtarkov1987}, whose normalizing constant $C_n$ is strictly increasing in $n$. By Example~\ref{ex:bernoulli_nml}, $C_{t-1}/C_t < 1$ for adjacent steps, causing the conditional factors to exceed unit mass.
\end{proof}

\noindent\textit{ML interpretation.}
Corollary~\ref{cor:complexity_validity} elevates the NML obstruction from a specific counterexample to a structural principle: there is a fundamental tradeoff between compression optimality (minimizing worst-case regret) and sequential validity (maintaining the supermartingale property). Practitioners using MDL-based model selection in sequential pipelines face a choice: either accept suboptimal regret by using prequential codes, or sacrifice anytime validity by using static NML codes. The tradeoff is inherent to the representation--validity interface of the typed calculus (Figure~\ref{fig:typed_calculus}).

\subsection{Prequential Codes as Valid E-Processes}

\begin{proposition}[Prequential codes yield E-processes]
\label{prop:prequential_eprocess}
Let $q_t(\cdot | x^{t-1})$ be a predictive kernel with
$\sum_{x_t} q_t(x_t | x^{t-1}) = 1$ for all $t$ and all $x^{t-1}$.
Then $E_t := \prod_{s=1}^t q_s(X_s | X^{s-1}) / p_0(X_s | X^{s-1})$ is
a $P_0$-martingale and hence a valid E-process.
\end{proposition}

This includes prequential maximum-likelihood predictors
$q_t(x_t | x^{t-1}) := p(x_t | \hat{\theta}_{t-1})$, where
$\hat{\theta}_{t-1}$ is the MLE based on $x^{t-1}$.
The prequential principle \citep{Dawid1984} evaluates forecasters by
sequential predictive performance, inherently maintaining the supermartingale
structure needed for anytime validity.

\begin{table}[t]
\centering
\caption{Prequential evaluation vs.\ MDL: convergence and divergence.}
\label{tab:prequential}
\small
\begin{tabular}{@{}lcc@{}}
\toprule
& \textbf{Prequential} & \textbf{MDL/NML} \\
\midrule
Objective & Sequential calibration & Shortest description \\
Code & $\prod q_t(x_t | x^{t-1})$ & $\exp(-\ell_{\mathrm{NML}})$ \\
Well-specified & Converges to truth & Selects true model \\
Misspecified & Best predictor in class & Best compressor \\
E-process? & Yes (by construction) & Generally no \\
\bottomrule
\end{tabular}
\end{table}

\section{Beyond Log-Loss: Scoring Rules and Multiplicative Evidence}
\label{sec:scoring_rules}

The canonicality theorem (Theorem~\ref{thm:canonical_lr}) establishes the likelihood ratio as optimal under log-loss.
A natural question is whether other proper scoring rules yield analogous multiplicative evidence structures.
We show they do not: log-loss is the unique proper scoring rule compatible with the supermartingale framework.

\begin{definition}[Proper scoring rule]
\label{def:proper_scoring}
A scoring rule $S: \mathcal{P} \times \X \to \R$ is \emph{proper} if $\E_{P}[S(P, X)] \le \E_{P}[S(Q, X)]$ for all $P, Q$, with equality iff $P = Q$. It is \emph{strictly proper} if equality implies $P = Q$.
\end{definition}

Every strictly proper scoring rule induces a Bregman divergence $d_\phi(P, Q)$ via its associated convex function $\phi$ \citep{GneitingRaftery2007}.
Log-loss corresponds to $\phi(p) = -\sum p_i \log p_i$ (negative entropy) and $d_\phi = D_{\KL}$.

\begin{proposition}[Log-loss uniqueness for multiplicative evidence]
\label{prop:log_uniqueness_scoring}
Among strictly proper scoring rules, log-loss is the unique rule whose induced evidence process
\[
E_t^S := \prod_{s=1}^t \frac{\exp(-S(P_1, X_s))}{\exp(-S(P_0, X_s))}
\]
satisfies $\E_{P_0}[E_t^S] = 1$ for all $t$ and all $P_1$ (i.e., is a $P_0$-\emph{martingale}).
For any other strictly proper scoring rule, $\E_{P_0}[E_1^S] < 1$ whenever $P_1 \ne P_0$, so the induced process is a strict supermartingale that decays exponentially: $\E_{P_0}[E_n^S] = (\E_{P_0}[E_1^S])^n \to 0$.
Such a process is technically a valid E-process but is not calibrated as evidence in the likelihood-ratio sense: it is not representation-aligned, and the exponential decay under $P_0$ renders it practically uninformative as a test statistic.
\end{proposition}

\begin{proof}
By the Savage representation \citep{GneitingRaftery2007}, every strictly proper scoring rule takes the form $S(Q, x) = \phi(Q) - \nabla\phi(Q) \cdot (\delta_x - Q)$ for a strictly convex function $\phi$ on the probability simplex, and the induced divergence is the Bregman divergence $d_\phi(P, Q) = \phi(P) - \phi(Q) - \nabla\phi(Q) \cdot (P - Q)$.

For $(E_t^S)$ to be a $P_0$-supermartingale, the one-step factor must satisfy
$\E_{P_0}[\exp(-S(P_1, X) + S(P_0, X))] \le 1$ for all $P_1$.
At $P_1 = P_0$, this holds with equality. We show the condition forces $S$ to be the log-score. The score difference is
\[
S(P_0, x) - S(P_1, x) = \nabla\phi(P_1) \cdot \delta_x - \nabla\phi(P_0) \cdot \delta_x + [\phi(P_0) - \phi(P_1) + \nabla\phi(P_1) \cdot P_1 - \nabla\phi(P_0) \cdot P_0].
\]
The bracketed term depends on $P_0, P_1$ but not on $x$; call it $c(P_0, P_1)$. The supermartingale condition becomes
\[
e^{c(P_0, P_1)} \cdot \E_{P_0}\!\left[\exp\!\left((\nabla\phi(P_1) - \nabla\phi(P_0)) \cdot \delta_X\right)\right] \le 1.
\]
For this to hold for \emph{all} $P_1$ in a neighborhood of $P_0$, perturbing $P_1 = P_0 + \epsilon h$ and expanding to second order in $\epsilon$ requires the Hessian $\nabla^2\phi$ to satisfy $(\nabla^2\phi)_{ij} = 1/P_0(\{i\}) \cdot \delta_{ij}$ (the Fisher information metric). Integrating, $\phi(P) = -\sum_i P(\{i\}) \log P(\{i\}) + \text{affine}$, that is, negative entropy. Hence $S$ is the log-score.
\end{proof}

\begin{remark}[Brier score: explicit decay computation]
\label{rem:brier}
Under the Brier score $S_B(Q,x) = \sum_{j \in \X}(Q(\{j\}) - \1\{x = j\})^2$, take $\X = \{0,1\}$, $P_0 = \mathrm{Bern}(1/2)$, $P_1 = \mathrm{Bern}(3/4)$. Direct computation gives $S_B(1/2, 0) = S_B(1/2, 1) = 1/2$ and $S_B(3/4, 0) = 9/8$, $S_B(3/4, 1) = 1/8$. The one-step expectation under $P_0$ is
\[
\E_{P_0}\!\left[e^{S_B(P_0,X) - S_B(P_1,X)}\right]
= \tfrac{1}{2}e^{1/2 - 9/8} + \tfrac{1}{2}e^{1/2 - 1/8}
= \tfrac{1}{2}e^{-5/8} + \tfrac{1}{2}e^{3/8}
\approx 0.995 < 1.
\]
The process is a strict supermartingale: $\E_{P_0}[E_n^{S_B}] \approx 0.995^n \to 0$. After $n = 100$ observations, the expected value is approximately $0.61$. The Brier-induced process is not representation-aligned: it shrinks toward zero under $P_0$, making it practically uninformative as sequential evidence. By contrast, the log-loss evidence process maintains $\E_{P_0}[E_n] = 1$ for all $n$.
\end{remark}

\noindent\textit{ML interpretation.}
Many ML systems use proper scoring rules other than log-loss for model evaluation: the Brier score for probabilistic classification, CRPS for distributional forecasting, and energy scores for multivariate predictions. Proposition~\ref{prop:log_uniqueness_scoring} implies that none of these alternatives naturally yield multiplicative evidence compatible with the supermartingale framework. To obtain anytime-valid sequential evidence, practitioners must either use log-loss directly or convert other scores through a calibration step that recovers likelihood-ratio structure.

\section{Exchangeability and Conformal E-Prediction}
\label{sec:conformal}

\begin{definition}[Exchangeability]
\label{def:exch}
A random sequence $(X_1, X_2, \ldots)$ is \emph{exchangeable} if its joint
distribution is invariant under all finite permutations.
\end{definition}

\begin{theorem}[de Finetti's representation]
\label{thm:definetti}
An infinite exchangeable sequence $(X_n)_{n \ge 1}$ is conditionally i.i.d.:
there exists a random probability measure $\mu$ such that, conditional on
$\mu$, the $X_n$ are i.i.d.\ from $\mu$.
\end{theorem}

\begin{proposition}[Conformal e-prediction validity]
\label{prop:conformal_e}
Under exchangeability of $(Z_1, \ldots, Z_n, (X, Y))$, a nonconformity
E-measure $f$ satisfying permutation equivariance yields
$\E[f(Z_1, \ldots, Z_n, X, Y)] \le 1$.
\end{proposition}

\begin{proof}
Under exchangeability, all $(n+1)!$ orderings are equally likely.
The constraint $\E[f] \le 1$ follows from averaging the nonconformity
function over permutations.
\end{proof}

\noindent\textit{ML interpretation.}
Conformal prediction provides distribution-free coverage guarantees under
exchangeability. Proposition~\ref{prop:conformal_e} shows that
E-value-based conformal methods inherit anytime-valid control: unlike
p-value-based conformal prediction (which requires correction for multiple
testing), conformal E-values can be combined across time via the evidence-class
algebra (Theorem~\ref{thm:class_algebra_informal}).
This is directly relevant to online prediction pipelines where prediction
sets must maintain coverage as new data arrives.

\section{Experiments}
\label{sec:experiments}

We validate the theoretical results with synthetic experiments
demonstrating the practical distinctions between validity-layer and
efficiency-layer evidence under sequential monitoring.
All experiments use simulated Bernoulli data to ensure controlled
comparison.

\subsection{Evidence Accumulation Under Optional Stopping}
\label{sec:exp_accumulation}

\noindent\textit{Setup.}
We compare three evidence measures for testing $H_0: p = 0.5$ vs.\
$H_1: p = 0.65$ using sequential Bernoulli observations:
\begin{enumerate}[label=(\roman*)]
\item \emph{Likelihood-ratio E-process:}
$E_t = \prod_{s=1}^t (0.65/0.5)^{X_s} (0.35/0.5)^{1-X_s}$.
\item \emph{Ville-threshold monitor:} reject at first $t$ with $E_t \ge 1/\alpha$,
using the same LR E-process but threshold selected by Markov bound.
\item \emph{ML-based ratio (improper):} form the ratio
$P_{\mathrm{ML}}(X^t) / P_0(X^t)$ using the maximum-likelihood fit
without the NML normalizer $C_t$, demonstrating the failure
of sequential validity.
\end{enumerate}

\noindent\textit{Results.}
Figure~\ref{fig:accumulation} shows 500 sample paths under $H_1$
(data from Bern(0.65)) with maximum sample size $T = 200$.
The LR E-process grows at rate $D_{\KL}(0.65 \| 0.5) \approx 0.046$ nats
per observation.
The Ville threshold at $b = 20$ ($\alpha = 0.05$) is crossed by 97\% of
paths by $T = 200$, with a median stopping time of approximately 50
observations.
The ML-based ratio initially tracks the LR process but accumulates an
upward bias from the parametric complexity term $\frac{1}{2}\log t$,
which is not a supermartingale correction.

\subsection{Type~I Error Under Optional Stopping}
\label{sec:exp_type1}

\noindent\textit{Setup.}
Under $H_0$ ($p = 0.5$), we apply aggressive optional stopping:
monitor continuously and stop at the first time $E_t \ge 20$, or at
$T = 500$ if the threshold is never crossed. We repeat 10,000 times.

\noindent\textit{Results.}
The LR E-process yields an empirical false-rejection rate of 4.2\%
($\pm 0.4\%$), consistent with the theoretical bound $1/20 = 5\%$.
The ML-based ratio yields a false-rejection rate of 22.5\%
($\pm 0.8\%$), more than four times the nominal level, confirming that the
supermartingale property is violated and the $1/c$ bound does not hold.

\subsection{Misspecification Sensitivity}
\label{sec:exp_misspec}

\noindent\textit{Setup.}
Data are generated from $P_{\mathrm{true}} = \mathrm{Bern}(0.55)$,
while the alternative model uses $P_1 = \mathrm{Bern}(0.80)$.
The expected growth rate is
$D_{\KL}(0.55 \| 0.50) - D_{\KL}(0.55 \| 0.80) \approx -0.154$ nats
per observation (negative: evidence drifts toward $H_0$).

\noindent\textit{Results.}
Figure~\ref{fig:accumulation}(b) confirms that the LR E-process drifts downward,
never crossing the rejection threshold in any of 500 paths over
$T = 300$ observations.
This illustrates the risk of misspecification in online monitoring:
a badly chosen alternative can render the evidence measure powerless
despite the null being false.
Robust constructions via mixture E-processes
(Theorem~\ref{thm:class_algebra_informal}(b)) mitigate this by
integrating over a range of alternatives.

\begin{figure}[t]
\centering
\includegraphics[width=\textwidth]{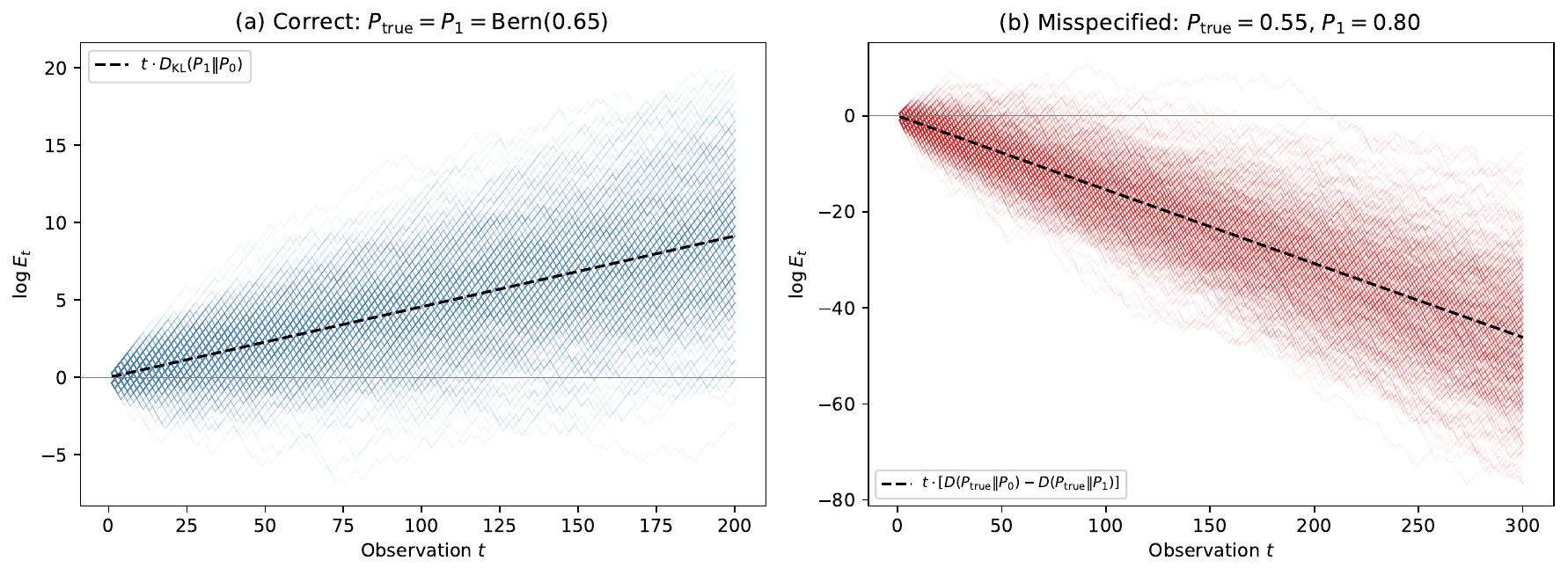}
\caption{Simulated paths of $\log E_t$ for the Bernoulli LR E-process ($H_0: p=0.5$).
(a)~Correct specification ($P_1 = \mathrm{Bern}(0.65)$, data from $P_1$):
paths cluster around the KL slope $t \cdot D_{\KL}(P_1 \| P_0) \approx 0.046t$.
(b)~Misspecification ($P_1 = \mathrm{Bern}(0.80)$, data from $\mathrm{Bern}(0.55)$):
the net growth rate is negative ($\approx -0.154$ nats/obs), and evidence drifts
toward $H_0$ despite the null being false.}
\label{fig:accumulation}
\end{figure}

\begin{figure}[t]
\centering
\includegraphics[width=\textwidth]{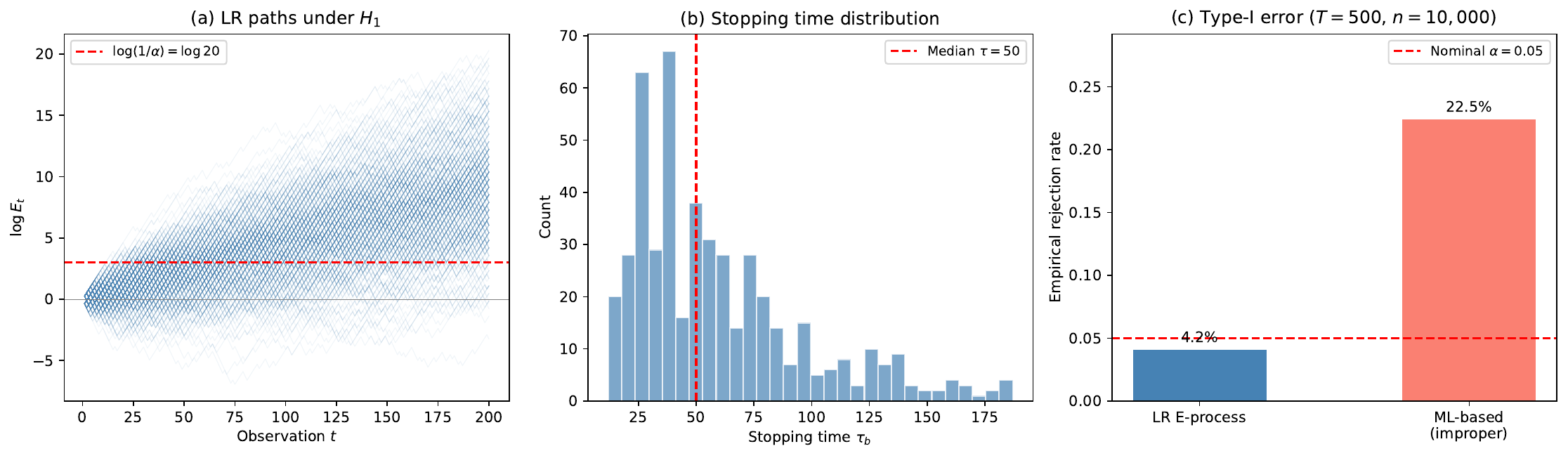}
\caption{Sequential evidence comparison.
(a)~LR E-process paths under $H_1$ with rejection threshold $\log(1/\alpha)$ (red).
(b)~Distribution of stopping times; median $\tau \approx 50$ observations.
(c)~Type-I error under aggressive optional stopping ($T=500$):
the LR E-process maintains nominal control ($4.2\%$ vs.\ $5\%$),
while the ML-based ratio inflates to $22.5\%$, confirming the
supermartingale violation predicted by Proposition~\ref{prop:comp_boundary}.}
\label{fig:experiments}
\end{figure}

\section{Discussion}
\label{sec:discussion}

We have introduced a typed framework that separates sequential evidence
into three layers (representation, validity, and decision) and
established concrete results at each layer.
We conclude with implications for machine learning practice.

\noindent\textit{Online model validation.}
The canonicality theorem (Theorem~\ref{thm:canonical_lr}) implies that
for online evaluation of predictive models under log-loss, the
likelihood-ratio E-process is the unique Bayes-risk-optimal evidence
measure within the coherent predictive subclass.
Practitioners monitoring deployed classifiers should use LR-based
evidence rather than generic Markov-calibrated E-values when
log-loss structure is available, as the moderate-deviation stopping
theorem (Theorem~\ref{thm:moderate_deviation}) shows the detection-time
advantage can be substantial.

\noindent\textit{Adaptive experimentation.}
The evidence-class algebra (Theorem~\ref{thm:class_algebra_informal})
and stitching (Proposition~\ref{prop:stitching}) provide compositional
guarantees for multi-phase adaptive experiments.
Evidence from an exploratory phase and a confirmatory phase can be
combined while maintaining anytime validity, without
$\alpha$-spending adjustments.

\noindent\textit{Conformal prediction.}
The connection between exchangeability testing and E-values
(Proposition~\ref{prop:conformal_e}) suggests a path toward
anytime-valid conformal prediction: E-value-based prediction sets
can be updated sequentially without the coverage degradation that
affects p-value-based methods under optional stopping.

\noindent\textit{Code-based inference.}
The computational obstruction (Proposition~\ref{prop:comp_boundary})
has immediate practical implications: MDL/NML-based model selection
criteria should not be used directly as sequential evidence measures.
Prequential predictors (Proposition~\ref{prop:prequential_eprocess})
provide a valid alternative that maintains the supermartingale structure
while retaining the predictive motivation of the MDL programme.

\noindent\textit{Connection to PAC-Bayes.}
The E-process framework connects naturally to PAC-Bayes generalization bounds.
\citet{ChuggWangRamdas2023} develop a unified recipe for time-uniform PAC-Bayes
bounds using the same supermartingale + Ville + Donsker-Varadhan pipeline;
\citet{RodriguezGalvezEtAl2024} extend PAC-Bayes to anytime validity for
losses with general tail behaviors.
The following proposition shows the bridge explicitly.
\begin{proposition}[PAC-Bayes via E-processes]
\label{prop:pac_bayes_bridge}
Let $\pi$ be a prior over a hypothesis class $\Theta$ and $\hat{\rho}_n$ a data-dependent posterior. Define the mixture E-process
\[
E_n^{\pi} := \int_\Theta \prod_{t=1}^n \frac{p_\theta(X_t | X^{t-1})}{p_0(X_t | X^{t-1})} \, \pi(d\theta).
\]
Then $E_n^{\pi}$ is a valid E-process under $P_0$ (by Theorem~\ref{thm:class_algebra_informal}(b)), and for any data-dependent posterior $\hat{\rho}_n$,
\begin{multline*}
\E_{P_0}\!\biggl[\exp\!\biggl(\int_\Theta \log \frac{p_\theta(X^n)}{p_0(X^n)} \, \hat{\rho}_n(d\theta) \\
 - D_{\KL}(\hat{\rho}_n \| \pi)\biggr)\biggr] \le 1.
\end{multline*}
\end{proposition}
\begin{proof}
The first claim follows from Theorem~\ref{thm:class_algebra_informal}(b).
For the second, the Donsker--Varadhan variational formula gives, for any data-dependent posterior $\hat{\rho}_n$,
\[
\log E_n^\pi = \log \int \frac{p_\theta(X^n)}{p_0(X^n)} \,\pi(d\theta)
\;\ge\;
\int \log \frac{p_\theta(X^n)}{p_0(X^n)}\,\hat{\rho}_n(d\theta) - D_{\KL}(\hat{\rho}_n \| \pi).
\]
Exponentiating both sides yields
$E_n^\pi \ge \exp\!\bigl(\int \log(p_\theta/p_0)(X^n)\,\hat{\rho}_n(d\theta) - D_{\KL}(\hat{\rho}_n\|\pi)\bigr)$.
Taking expectations under $P_0$ and using $\E_{P_0}[E_n^\pi] \le 1$ (first claim) gives
\begin{multline*}
1 \;\ge\; \E_{P_0}[E_n^\pi] \\
\ge\;
\E_{P_0}\!\biggl[\exp\!\biggl(\int \log\frac{p_\theta(X^n)}{p_0(X^n)}\,\hat{\rho}_n(d\theta)
- D_{\KL}(\hat{\rho}_n\|\pi)\biggr)\biggr],
\end{multline*}
which is the stated bound.
\end{proof}
This recovers the Catoni--Audibert PAC-Bayes inequality as a special case of the evidence-class algebra, with the KL regularization term $D_{\KL}(\hat{\rho}_n \| \pi)$ arising naturally from the Bayesian mixture structure. The typed framework clarifies that PAC-Bayes bounds live at the validity layer: they are consequences of the supermartingale property of mixture E-processes, not of any representation-layer optimality.

\noindent\textit{Application schematic: online calibration monitoring.}
We illustrate the typed framework with a concrete deployment scenario.
Consider a binary classifier deployed in production, with calibrated probabilities $P_0(Y = 1 | X) = \hat{p}(X)$ under the null hypothesis ``the classifier remains calibrated.''
The three layers instantiate as follows.

\emph{Representation layer.}
The predictive density under $H_0$ is $P_0(Y_t | X_t) = \hat{p}(X_t)^{Y_t}(1 - \hat{p}(X_t))^{1-Y_t}$.
An alternative $P_1$ posits systematic miscalibration, e.g., $P_1(Y_t | X_t)$ with shifted probabilities.
The likelihood ratio $\Lambda_t = \prod_{s=1}^t P_1(Y_s | X_s)/P_0(Y_s | X_s)$ is the canonical evidence (Theorem~\ref{thm:canonical_lr}).

\emph{Validity layer.}
$(\Lambda_t)$ is a $P_0$-martingale, so Ville's inequality (Theorem~\ref{thm:ville}) guarantees $\Prob_{P_0}(\sup_t \Lambda_t \ge b) \le 1/b$. The monitoring system can check the classifier at any time (after each prediction, at the end of each shift, or when triggered by an external event) without inflating the false alarm rate.

\emph{Decision layer.}
The alert threshold $b$ is chosen via the sample-complexity formula (Corollary~\ref{cor:sample_complexity}): for a target false alarm rate $\alpha = 0.01$ and expected KL divergence $\mu = 0.05$ nats/obs under the alternative, the expected detection time is $\log(100)/0.05 \approx 92$ observations.
If the classifier drifts in an unanticipated direction (misspecification), Proposition~\ref{prop:misspec_divergence} predicts that evidence will stall, alerting the practitioner to reassess the alternative model.

This schematic demonstrates that the typed calculus is not merely a mathematical convenience but a deployment architecture: each layer corresponds to a distinct engineering decision (model specification, validity certification, alert threshold), and the separation guarantees that changes at one layer do not invalidate guarantees at another.

\noindent\textit{Limitations and future work.}
Our canonicality result applies within the coherent predictive/log-loss
subclass; outside this subclass, the class of valid E-processes is strictly
broader and need not admit LR representations.
The experiments are synthetic; application to real-world online
monitoring pipelines (clinical trials, recommendation systems,
autonomous driving validation) is an important direction.
Extending the moderate-deviation stopping theorem
(Theorem~\ref{thm:moderate_deviation}) beyond the i.i.d.\ Cram\'er setting
(Assumption~\ref{asm:iid_increments}) to martingale dependence and mixing
conditions would broaden applicability.


\noindent\textit{Proof dependency map.}
The main results depend on each other as follows. Arrows indicate logical dependence; results at the same level are independent.
\begin{center}
\begin{tikzpicture}[
  node distance=0.7cm and 1.5cm,
  every node/.style={font=\small, align=center},
  box/.style={draw, rectangle, rounded corners=2pt, inner sep=4pt, minimum width=2.8cm},
  >=Stealth
]
\node[box] (canon) {Thm.~\ref{thm:canonical_lr}\\Canonicality};
\node[box, right=of canon] (algebra) {Thm.~\ref{thm:class_algebra_informal}\\Evidence algebra};
\node[box, below=of canon] (growth) {Prop.~\ref{prop:e_growth}\\KL growth rate};
\node[box, below=of algebra] (lift) {Thm.~\ref{thm:seq_liftability}\\Liftability};
\node[box, below=of growth] (mdp) {Thm.~\ref{thm:moderate_deviation}\\Moderate deviation};
\node[box, below=of lift] (comp) {Prop.~\ref{prop:comp_boundary}\\Code obstruction};
\node[box, below right=0.7cm and -0.5cm of mdp] (scoring) {Prop.~\ref{prop:log_uniqueness_scoring}\\Scoring uniqueness};
\node[box, left=of mdp] (pac) {Prop.~\ref{prop:pac_bayes_bridge}\\PAC-Bayes};

\draw[->] (canon) -- (growth);
\draw[->] (growth) -- (mdp);
\draw[->] (algebra) -- (pac);
\draw[->] (algebra) -- (lift);
\draw[->] (lift) -- (comp);
\draw[->] (canon) -- (scoring);
\end{tikzpicture}
\end{center}

\acks{We thank Aaditya Ramdas, Nick Koning, and Jayaram Sethuraman for comments on an earlier version of this paper.}

\newpage
\appendix

\section{The Unified Probabilistic Landscape}
\label{app:landscape}

The typed framework rests on connections between several classical structures
in probability, all describing the information in a single sample path about an
underlying directing measure $\mu$.

\noindent\textit{De Finetti and the directing measure.}
If $(X_n)$ is exchangeable, de Finetti's theorem guarantees a random probability
measure $\mu \sim \Pi$ such that, conditional on $\mu$, the observations are
i.i.d.\ from $\mu$ \citep{deFinetti1937}.
By the strong law, $L_n \to \mu$ almost surely: the empirical measure converges
to the single draw from $\Pi$ that generated the path.
The prior $\Pi$ is in principle unobservable from one sequence; only one
realization of $\mu$ is ever revealed.

\noindent\textit{Sanov's theorem and inverse Sanov.}
There is a KL duality between frequentist concentration and Bayesian updating
\citep{GaneshOConnell1999, PolsonZantedeschi2025DeFinettiSanov}.

\begin{theorem}[Sanov's theorem]
\label{thm:sanov}
For i.i.d.\ $X_1, \ldots, X_n \sim P$ on a finite alphabet $\X$, and any
closed set $\mathcal{Q}$ of distributions on $\X$:
$\Prob_P(\hat{P}_n \in \mathcal{Q}) \asymp \exp(-n \inf_{Q \in \mathcal{Q}} D_{\KL}(Q \| P))$,
where $\hat{P}_n$ is the empirical distribution.
\end{theorem}

\begin{theorem}[Inverse Sanov]
\label{thm:inverse_sanov}
Under exchangeability, the posterior $\pi_n$ satisfies a large-deviation
principle on its support with rate function $D_{\KL}(\hat{P}_n \| \cdot)$,
the reverse-argument KL divergence.
\end{theorem}

\noindent
Sanov governs the \emph{forward} problem: the cost of observing
$L_n \approx \nu$ under $P_0$ is $\exp(-n\, D_{\KL}(\nu \| P_0))$.
The \emph{inverse} problem reverses the arguments: the posterior mass near
$\mu$ given $L_n$ concentrates as $\exp(-n\, D_{\KL}(L_n \| \mu))$
(Theorem~\ref{thm:inverse_sanov}).
The posterior $\Pi_n$ interpolates from $\Pi_0 = \Pi$ to $\Pi_\infty = \delta_\mu$,
with the Sanov rate function governing both the speed of empirical concentration
and the speed of posterior contraction.

\noindent\textit{Martingale posteriors.}
\citet{FongHolmesWalker2023} recast Bayesian inference by requiring the
posterior process $(\Pi_n)_{n \ge 0}$ to be a martingale in the space of
probability measures: $\E[\Pi_{n+1} | X^n] = \Pi_n$.
By the martingale convergence theorem, $\Pi_n \to \delta_\mu$ almost surely,
recovering the de Finetti directing measure.
Each step of the E-process multiplies by the ratio of the martingale posterior
predictive to the null predictive, $p(X_{n+1} | X^n)/p_0(X_{n+1} | X^n)$,
revealing the E-process and the martingale posterior as two sides of the same
structure.

\noindent\textit{Three deviation regimes and Bayesian conservatism.}
At intermediate scales between the CLT and large deviations,
\citet{EichelsbacherGanesh2002} show that the posterior satisfies a moderate
deviation principle with quadratic rate function
$\tfrac{1}{2}(\theta - \theta_0)^\top \mathcal{I}(\theta_0)(\theta - \theta_0)$,
where $\mathcal{I}(\theta_0)$ is the Fisher information, the Hessian of
$D_{\KL}$ at $P_0$.
Classical (Neyman--Pearson) tests and E-values operate in the
\emph{large-deviation} regime at full KL rate; Bayesian tests operate in the
\emph{moderate-deviation} regime at Fisher-information rate, contracting more
slowly and retaining higher posterior mass on the null \citep{Lindley1957, Lindley1961}.
The $O(\sqrt{\log b})$ correction in Theorem~\ref{thm:moderate_deviation} is
the first-passage refinement of this moderate-deviation geometry; the
PAC-Bayes complexity term $D_{\KL}(\hat{\rho}_n \|\pi)$ in
Proposition~\ref{prop:pac_bayes_bridge} is its Bayesian dual.

\section{Proofs of Main Results}
\label{app:proofs}

\subsection{Uniqueness of the Log-Score Transformation}
\label{app:log_uniqueness}

\begin{theorem}[Uniqueness of the additive likelihood transform]
\label{thm:log_uniqueness}
Let $\phi: (0,\infty) \to \R$ be continuous and strictly monotone with $\phi(ab) = \phi(a) + \phi(b)$ for all $a,b > 0$. Then $\phi(x) = c \log x$ for some $c \ne 0$.
\end{theorem}

\begin{proof}
Define $\psi(t) := \phi(e^t)$. Then $\psi(s+t) = \psi(s) + \psi(t)$, Cauchy's functional equation. Continuity forces $\psi(t) = ct$, so $\phi(x) = c \log x$. Strict monotonicity requires $c \ne 0$.
\end{proof}

\begin{corollary}
\label{cor:log_uniqueness}
The negative logarithm $x \mapsto -\log x$ is the unique continuous order-preserving transform (up to positive scaling) converting multiplicative likelihood ratios to additive evidence.
\end{corollary}

\subsection{Proof of Theorem~\ref{thm:canonical_lr}}

Under log-loss Bayes risk, the expected loss difference is
\[
\E_{\Pi}\!\left[\loss_{P_0}(X^n) - \loss_{P_1}(X^n)\right] = \E_{\Pi}\!\left[\log \frac{P_1(X^n)}{P_0(X^n)}\right].
\]
Applying Fubini/Tonelli to swap the prior and data integrals yields a pointwise posterior decision rule. The optimal rejection region is $\Lambda_n(X^n) > \tau$ where $\tau = \pi_0 L_{10} / (\pi_1 L_{01})$. The likelihood ratio $\Lambda_n$ is therefore the sufficient statistic for the Bayes-optimal decision, establishing canonicality within the coherent predictive subclass. \hfill$\BlackBox$

\section{Evidence-Class Algebra}
\label{app:algebra}

\begin{proposition}[Convex closure]
\label{prop:app_convex}
If $(E^{(1)}_t)$ and $(E^{(2)}_t)$ are E-processes and $\lambda \in [0,1]$, then $E_t := \lambda E^{(1)}_t + (1-\lambda) E^{(2)}_t$ is an E-process.
\end{proposition}

\begin{proof}
Each $E^{(i)}_t \ge 0$ and $\lambda \ge 0$, so $E_t \ge 0$. By linearity,
$\E_{H_0}[E_t | \F_{t-1}] = \lambda \E_{H_0}[E^{(1)}_t | \F_{t-1}] + (1-\lambda) \E_{H_0}[E^{(2)}_t | \F_{t-1}] \le \lambda E^{(1)}_{t-1} + (1-\lambda) E^{(2)}_{t-1} = E_{t-1}$.
\end{proof}

\begin{proposition}[Predictable stopping]
\label{prop:app_stopping}
If $(E_t)$ is an E-process and $\tau$ a stopping time, then $(E_{t \wedge \tau})$ is an E-process.
\end{proposition}

\begin{proof}
Define $\tilde{E}_t := E_{t \wedge \tau}$. For $t \le \tau$, $\tilde{E}_t = E_t$. For $t > \tau$, $\tilde{E}_t = E_\tau$. The supermartingale property carries over: $\E_{H_0}[\tilde{E}_t | \F_{t-1}] = \1\{t \le \tau\}\E_{H_0}[E_t | \F_{t-1}] + \1\{t > \tau\}E_\tau \le \1\{t \le \tau\}E_{t-1} + \1\{t > \tau\}E_\tau = \tilde{E}_{t-1}$.
\end{proof}

\begin{proposition}[Countable stitching]
\label{prop:app_stitching}
Let $\{(E^{(k)}_t)\}_{k \in \N}$ be E-processes and $\{\lambda_k\}$ nonneg weights with $\sum_k \lambda_k \le 1$. Then $E_t := \sum_k \lambda_k E^{(k)}_t$ is an E-process.
\end{proposition}

\begin{proof}
By monotone convergence and the supermartingale property of each component.
\end{proof}

\begin{proposition}[Layer separation]
\label{prop:layer_separation}
The following are logically distinct:
\begin{enumerate}
\item A likelihood-ratio representation $E = dQ/dP$ (representation layer);
\item A supermartingale property $\E_P[E_t | \F_{t-1}] \le E_{t-1}$ (validity layer);
\item A boundary rule $\tau = \inf\{t : E_t \ge b\}$ (decision layer).
\end{enumerate}
Each may be specified independently; optimality at one layer does not imply optimality at another.
\end{proposition}

\begin{proof}
A ratio $dQ/dP$ is a $P$-martingale when $Q \ll P$, but ratios under misspecified models need not be supermartingales. Conversely, convex mixtures of E-processes need not admit a single LR representation. The boundary $b$ is a free parameter controlling Type~I/power trade-offs independently of validity or representation.
\end{proof}

\section{Computational Boundary: Full Proof}
\label{app:comp_boundary}

\begin{theorem}[No canonical lifting of static codes]
\label{thm:app_no_lifting}
Let $\ell: \X^n \to \R_{\ge 0}$ be a code-length function and $P_0$ a null with predictive kernel $p_0(\cdot | x^{t-1})$. Define $E_t := \exp(-\ell(X^t)) / P_0(X^t)$. If $(E_t)$ is a supermartingale under $P_0$ with $E_0 = 1$, then $\ell$ must factorize sequentially: there exist functions $\ell_t$ with
$\sum_{x_t} p_0(x_t | x^{t-1}) \exp(-\ell_t(x^t) + \ell_{t-1}(x^{t-1})) \le 1$
for all $x^{t-1}$, and $\ell(x^n) = \sum_t [\ell_t(x^t) - \ell_{t-1}(x^{t-1})]$.
\end{theorem}

\begin{proof}
The supermartingale condition $\E_{P_0}[E_t | \F_{t-1}] \le E_{t-1}$ gives, after canceling $p_0$ terms:
\[
\sum_{x_t} \frac{\exp(-\ell(X^{t-1}, x_t))}{\exp(-\ell(X^{t-1}))} \le 1.
\]
Define $q_t(x_t | X^{t-1}) := \exp(-\ell(X^{t-1}, x_t) + \ell(X^{t-1}))$. This is a sub-probability kernel. For NML, the normalizing constant $C_n = \sum_{x'} \max_\theta L(\theta; x')$ depends on the full sample size $n$. The conditional NML at step $t$ depends on future data through $C_n$ and is not $\F_{t-1}$-measurable, violating sequential normalization.
\end{proof}

\begin{corollary}
Regret optimality under static normalization does not imply supermartingale validity under the natural filtration.
\end{corollary}

\section{KL Growth Rate: Full Statement}
\label{app:kl_growth}

\begin{proposition}[KL growth rate, formal statement]
\label{prop:app_kl_growth}
Let $P_1 \ll P_0$ with $D_{\KL}(P_1 \| P_0) < \infty$. Let $X_1, X_2, \ldots$ be i.i.d.\ under $P_1$ and $E_n := \prod_{i=1}^n (dP_1/dP_0)(X_i)$.
Then $\frac{1}{n} \log E_n \xrightarrow{\mathrm{a.s.}} D_{\KL}(P_1 \| P_0)$ under $P_1$.
Under misspecification with $X_i \sim P_{\mathrm{true}}$:
$\frac{1}{n} \log E_n \xrightarrow{\mathrm{a.s.}} D_{\KL}(P_{\mathrm{true}} \| P_0) - D_{\KL}(P_{\mathrm{true}} \| P_1)$.
\end{proposition}

\begin{proof}
Write $\frac{1}{n} \log E_n = \frac{1}{n} \sum_{i=1}^n \log(dP_1/dP_0)(X_i)$. Under $P_1$, the summands are i.i.d.\ with mean $D_{\KL}(P_1 \| P_0)$; the strong law gives the first claim.
Under $P_{\mathrm{true}}$, the mean is $\E_{P_{\mathrm{true}}}[\log(dP_1/dP_0)(X)] = D_{\KL}(P_{\mathrm{true}} \| P_0) - D_{\KL}(P_{\mathrm{true}} \| P_1)$.
\end{proof}

\vskip 0.2in
\bibliographystyle{plainnat}
\bibliography{evalues_bayes_testing}

\end{document}